\documentclass[12pt, amsfonts, epsfig]{amsart}
\usepackage[foot]{amsaddr}
\usepackage{amsmath, amscd, amssymb}
\usepackage{graphicx, psfrag}
\usepackage[%
  width=16.5cm,
  left=3cm,
  right=3cm,
]{geometry}

\numberwithin{equation}{section}
\usepackage{bm}
\usepackage{stmaryrd}
\usepackage{graphpap, color}
\usepackage{mathrsfs}
\usepackage[shortlabels]{enumitem}
\usepackage{color}
\usepackage{cancel}
\usepackage{tikz-cd}
\usepackage{tikz}
\usepackage[mathscr]{eucal}
\usepackage{verbatim}
\usepackage[linktoc=all]{hyperref}
\hypersetup{
  colorlinks,
  citecolor=black,
  filecolor=black,
  linkcolor=black,
  urlcolor=black
}
\usepackage{latexsym}
\usepackage[all]{xy}

\def\slope{\mathrm{slope}}
\def\ZZ{\mathbb Z} 
\def\Irre{\mathrm{Irre}}
\def\Aut{\mathrm{Aut}}
\def\upmo{^{-1}}
\def\defeq{:=}

\def\beq{\begin{equation}}
\def\eeq{\end{equation}}

\def\sub{\subset}

\def\lra{\longrightarrow}

\def\sta{^\ast}

\def\sO{{\mathscr O}}
\def\sC{{\mathscr C}}

\def\PP{\mathbb P}
\def\CC{\mathbb C}
\xyoption{rotate}

\def\sO{{\mathscr O}}
\def\sB{{\mathscr B}}

\def\sL{{\mathscr L}}

\def\sO{\mathscr{O}}

\newcommand{\Si}{\Sigma}

\newtheorem{prop}{Proposition}[section]

\newtheorem{theo}[prop]{Theorem}
\newtheorem{lemm}[prop]{Lemma}
\newtheorem{coro}[prop]{Corollary}
\newtheorem{rema}[prop]{Remark}

\newtheorem{defi}[prop]{Definition}
\newcommand{\git}{/\!\!/}
\newtheorem{definition}[prop]{Definition}

\newtheorem{subl}[prop]{Sublemma}

\newtheorem{defi-prop}[prop]{Definition-Proposition}
\newtheorem{defi-theo}[prop]{Definition-Theorem}

\usepackage{lipsum}
\renewcommand{\tilde}{\widetilde}

\title[Stability condition in GLSM]{Stability conditions in the mathematical Gauged Linear Sigma Model}
\author{Huai-Liang Chang$^1$}
\author{Shuai Guo$^2$}
\author{Jun Li$^3$}
\author{Wei-Ping Li$^4$}
\author{Yang Zhou$^3$}

\address{$^1$ School of Mathematics and Statistics, Wuhan University.}
\address{$^2$ School of Mathematical Sciences, Peking University.}
\address{$^3$ Shanghai Center for Mathematical Sciences, Fudan University.}
\address{$^4$\ Department of Mathematics, The Hong Kong University of Science and Technology.}

\begin{document}
\maketitle
\begin{abstract}
  The theory of Mixed-Spin-P (MSP) fields was introduced by Chang-Li-Li-Liu for
  the quintic threefold, aiming at studying its higher-genus Gromov-Witten
  invariants. Chang-Guo-Li has successfully applied it to prove conjectures
  including the BCOV Feynman rule, Yamaguchi-Yau's polynomiality conjecture and
  the Holomorphic Anomaly Equation. Meanwhile, Fan-Jarvis-Ruan introduced a
  mathematical theory of Gauged Linear Sigma Model (GLSM), associating a counting
  theory to a GIT quotient with a super-potential, under suitable assumptions.

  This paper provides a common generalization of both works, by introducing
  new stability conditions in the mathematical GLSM. We show that our stability
  condition guarantees the separatedness and properness of the cosection
  degeneracy locus in the moduli.
  It generalizes the MSP fields construction to more general GIT quotients,
  including Calabi-Yau global complete intersections in toric varieties.
  This hopefully provides a
  geometric platform to effectively compute their higher-genus Gromov-Witten invariants.
\end{abstract}
\tableofcontents

\section{Introduction}
\label{sec:introduction}
The theory of Mixed-Spin-P (MSP) fields was developed in~\cite{chang2019mixed}.
It is a curve counting theory that bridges the Gromov-Witten
  theory of a hypersurface and its corresponding quantum singularity theory,
  also providing an effective algorithm for
  computing these Gromov-Witten invariants and FJRW invariants. As applications,
  it leads to a proof of the Bershadsky-Cecotti-Ooguri-Vafa (BCOV) Feynman sum
  conjecture for the Gromov-Witten invariants of quintic Calabi-Yau threefolds ,
  as well as the Yamaguchi-Yau polynomiality and the Holomorphic Anomaly
  Equation (HAE)
  (see the series of works \cite{chang2016effective, chang2018theory,
    chang2018polynomial, chang2018bcov}). 
The main motivation of this paper is to generalize MSP theory to a larger
class of Calabi-Yau threefolds, as a first step of studying the structure of
their higher-genus Gromov-Witten invariants.

The results of the paper applies to a wide class of target spaces, generalizing the theory
in \cite{fan2017mathematical}. It allows one to define more counting 
theories which are now often called Gauged Linear Sigma Model (GLSM) in the mathematics context.
We will discuss some examples other than MSP in Section~\ref{sec:examples}.
Besides the algebro-geometric construction, there are also constructions of
GLSM in symplectic geometry \cite{TianXuGLSM, TianXuVirtualCycle}.

\medskip 

In the remainder of this section, we briefly review the construction of MSP fields for the
quintic threefolds and its relation to \cite{fan2017mathematical}. Then we
introduce the main results of this paper, leaving the precise statements to the
next section.

Let $X\subset \mathbb P^4$ be any smooth quintic threefold, defined by a
homogeneous  polynomial
$F(x_1 ,\ldots, x_5)$. The first step towards
studying the higher genus Gromov-Witten invariants via MSP is to write $X$
as the critical locus of the function
\[
  p F: K_{\mathbb P^4} \longrightarrow \mathbb C,
\]
where $K_{\mathbb P^4}$ is the total space of the canonical bundle of $\mathbb
P^4$ and $p$ is the coordinate on the fibers. The function $pF$ is called a super-potential.
 Motivated by the P-fields in physics \cite{witten1993phases},
\cite{chang2020invariants} introduced the moduli $\overline{M}_{g}(\mathbb P^4,
d)^{p}$ of stable maps to $\mathbb P^4$ with P-fields. Concretely, it parametrizes tuples 
\[
  \xi = (\sC,  \mathscr L, \varphi, \rho),
\]
where
\begin{itemize}
\item
  $\sC$ is a nodal curve of genus $g$,
\item
  $\mathscr L$ is a line bundle on $\mathscr C$ of degree $d$,
\item
  $\varphi = (\varphi_1, \ldots, \varphi_5) \in H^0(\mathscr C, \mathscr
  L^{\oplus 5})$, $\rho \in H^0(\mathscr L^{\otimes (-5)} \otimes
  \omega_{\mathscr C})$, where $\omega_{\mathscr C}$ is the log dualizing sheaf
  of $\mathscr C$,
\end{itemize}
subject to the stability condition that
\[
  \sC \overset{\varphi}{\longrightarrow} \mathbb P^4
\]
is a stable map.

Note that if we did not have $\omega_{\mathscr C}^{\log}$ in
the definition of $\rho$ above, then the tuple $\xi$ would have been a stable
map
\[
  \sC \overset{(\varphi, \rho)}{\longrightarrow} K_{\mathbb P^4}.
\]
Thus, heuristically, $\xi$ is a stable map to $K_{\mathbb P^4}$ twisted by
$\omega_{\mathscr C}$ together with the $\mathbb C^*$-action on $K_{\mathbb
  P^4}$. More precisely, suppose $\{\mathscr U_{i}\}$ is an open cover of
$\mathscr C$ and  $s_i \in H^0(\mathscr U_i, \omega_{\mathscr C}|_{\mathscr
  U_i})$ are nonvanishing sections. Then, trivializing $\omega_{\mathscr
  C}|_{\mathscr U_i}$ by $s_i$, the restriction of $\xi$ to $\mathscr U_i$
becomes a map 
\[
  f_i: \mathscr U_i \longrightarrow K_{\mathbb P^4}.
\]
On the overlaps, $f_i$ and $f_j$ differ by the action of $s_i/s_j$ on the fibers
of $K_{\mathbb P^4}$. The $\mathbb C^*$-action on the fibers of $K_{\mathbb
  P^4}$ is called the $R$-charge.

For $g>0$, $\overline{M}_{g}(\mathbb
P^4, d)^{p}$ is not proper. However, by the cosection localization of Kiem-Li
\cite{kiem2013localizing},
a virtual cycle is constructed in 
the so-called cosection degeneracy locus. In
this case,  the cosection degeneracy locus is equal to
\[
  \overline{M}_{g}(X, d) = \{\xi \in \overline{M}_{g}(\mathbb
  P^4, d)^{p} \mid \rho\equiv 0,  F(\varphi) \equiv 0\}.
\]
Namely, it is the locus where the maps $f_i$ above all land in the critical locus
\[
  X = \mathrm{Crit}(pF) \subset K_{\mathbb P^4}.
\]
Note that this condition is independent of the choice of the local
trivializations $s_i$, since $X \subset K_{\mathbb P^4}$ is fixed by the $R$-charge.
It was shown in \cite{chang2020invariants} that the cosection localized virtual
cycle equals the virtual cycle in Gromov-Witten theory up to an explicit sign.

Replacing $(K_{\mathbb P^4}, pF)$ by $([\mathbb C^5 / \bm{\mu}_5], F)$,
where $\mu_5$ is the group of $5$th roots of unity acting on $\mathbb C^5$ by
scalar multiplication, the same construction above defines the
Fan-Jarvis-Ruan-Witten (FJRW) theory for
the singularity $F=0$ \cite{fan2013witten, chang2015witten}.

\medskip
Those two theories are closely related. Indeed,
$K_{\mathbb P^4}$ and $[\mathbb C^5/\bm{\mu}_5]$ are two GIT quotients of $\mathbb
C^{6}$ by $\mathbb C^*$, where the weights of the action are
\[
  (1,1,1,1,1, -5).
\]
Witten's discovery of the GLSM \cite[page~178]{witten1993phases} predicted the existence of a family of
theories interpolating between the Gromov-Witten theory of $X$ and the FJRW theory of
$([\mathbb C^5/\bm{\mu}_5], F)$. 
Inspired by Witten's insight, \cite{chang2019mixed} applied the 
variation of GIT (vGIT) master space construction (c.f.\
\cite[Section~3.1]{thaddeus1996geometric}) and defined MSP for the
quintic, which counts curves
``mapping'' into the master space
\[
  \mathcal Y =  \mathbb C^6 \times \mathbb C^2 \git_{\theta} \mathbb C^* \times \mathbb C^*,
\]
where
\begin{itemize}
\item
  the standard coordinates on $\mathbb C^6 \times \mathbb C^2$ denoted by $(x_1 ,\ldots, x_5, p, u, v)$,
\item
  the weights the $\mathbb C^* \times \mathbb C^*$-action are
  \[
    \begin{bmatrix}
      1 & 1 & 1 & 1 & 1 & -5 & 1 & 0 \\
      0 & 0 & 0 & 0 & 0 &  0 & 1 & 1
    \end{bmatrix}
  \]
\item
  the character $\theta$ is $\theta(t_1, t_2) = t_1t_2^2$, $t_1,t_2 \in \mathbb C^*$.
\end{itemize}

It is not hard to compute that the unstable locus is defined by
\[
  x_1 = \cdots = x_5 = u = 0, \quad u = v = 0, \quad p = u = 0.
\]
The space $\mathcal Y$ carries an additional $\mathbb C^*$ action only scaling the
$u$ coordinate. Thus, $\mathcal Y$ is a vGIT master space in the sense that both 
\begin{equation}
  \label{eq:KP4-and-C5-as-slices}
  K_{\mathbb P^4} \cong \{(u = 0)\} \subset \mathcal Y, \quad 
  [\mathbb C^5 / \bm{\mu}_5] \cong \{(v = 0)\} \subset \mathcal Y
\end{equation}
are among the fixed loci. 
Repeating the P-field construction using the space
$\mathcal Y$ yields the theory of MSP for the quintic \cite{chang2019mixed}. We will revisit
the details of MSP in Section~\ref{sec:brief}, as a special case of the main theorem of
this paper.

In parallel with the development of MSP, Fan-Jarvis-Ruan introduced a mathematical
theory of GLSM \cite{fan2017mathematical}, which applies the P-field
construction of \cite{chang2015witten, chang2020invariants}
to a broad class of GIT quotients, combined with the idea of $\epsilon$-stable
quasimaps \cite{ciocan2014stable}.
Their moduli parametrizes so-called LG-quasimaps. 
When taking the GIT quotient to be $K_{\mathbb P^4}$ or $[\mathbb C^5 / \bm{\mu}_5]$,
their theory are the same as those in \cite{chang2015witten, chang2020invariants}.

However, applying \cite{fan2017mathematical} to the master space $\mathcal Y$ does not
produce the MSP. One of the main difficulties of defining GLSM invariants for the
target space $\mathcal Y$ via the P-field construction is how to choose a
stability condition to guarantee the properness of the cosection degeneracy
locus. Such properness is crucial to defining invariants using intersection
  theory, and crucial to the master space technique via virtual localization.
The finite-automorphism condition in \cite{chang2019mixed} was ad
hoc and do not generalize well to other targets. In general it yields a
non-separated moduli space.
While the general stability in \cite{fan2017mathematical} requires a so-called good lift, which rarely
exists for vGIT master spaces like $\mathcal Y$.

In the absence of a good lift,  \cite{fan2017mathematical} only has
$\epsilon = 0^+$ stability, which allows the LG quasimap to meet the
unstable points \cite[Section~7.5.2]{fan2017mathematical}. In other words, the actual target space has to be the Artin stack quotient
$\mathfrak Y:= [\mathbb C^6 \times \mathbb C^2 / \mathbb C^* \times \mathbb
C^*]$ instead of $\mathcal Y$. 

For the purpose of studying higher genus Gromov-Witten invariants via MSP, 
a key difference is that, the two slices defined by $u = 0$ and $v=0$
intersect nontrivially in $\mathfrak Y$, in contrast to 
\eqref{eq:KP4-and-C5-as-slices}, where they do not intersect in $\mathcal Y$.
In the proof of the BCOV Feynman sum conjecture in
\cite{chang2018polynomial, chang2018bcov},
the authors utilized the disjointness of the copies of $K_{\mathbb P^4}$ and
$[\mathbb C^5/\bm{\mu}_5]$ in $\mathcal Y$ (see \eqref{eq:KP4-and-C5-as-slices}
above) to
separate their contributions to the MSP invariants. Moreover, the critical locus of the
super-potential $pF: \mathcal Y \to \mathbb C$ is equal to
\[
  \{(F = p = 0)\} \cup \{x_1 = \cdots = x_5 = 0\}
\]
Note that by the stability condition, $v$ must be nonvanishing on
\[
  C_{X}:= \{(F = p = 0)\}\subset \mathcal Y,
\]
and $(x_1 ,\ldots, x_5, u)$ defines an embedding $C_X \hookrightarrow\mathbb P^5$.
It is not hard to see that $C_X$ is isomorphic to the projective
cone $C_{X}$ over the quintic $X$ in $\mathbb P^5$. Counting curves in
$C_{X}$ is referred to as the $[0,1]$-theory in \cite{chang2018polynomial, chang2018bcov}.
In N-MSP, a variant of MSP where $\mathbb P^5$ is replaced by $\mathbb P^{4+N}$,
counting certain rational curves in $C_{X}$ gives the propagators in the BCOV Feynman sum.
It is
not clear to us how to generalize those ideas working instead with Artin stack
like $\mathfrak Y$, where the closure of
$K_{\mathbb P^4}$ and $[\mathbb C^5/\bm{\mu}_5]$
intersect nontrivially.

In this paper we define a new stability for GLSM, called $\Omega$-stability,
which refines the quasimap stability by introducing a
notion of slopes incorporating the so-called $R$-charges.
The GLSM defined using $\Omega$-stability will cover both \cite{chang2019mixed}
and \cite{fan2017mathematical}. In the example of $K_{\mathbb
P^4}$ discussed above, the $R$-charge is reflected in appearance of
$\omega_{\mathscr C}$ in
\[
  \rho \in H^0(\mathscr L^{\otimes (-5)} \otimes \omega_{\mathscr C}).
\]
Hence we chose the Greek letter omega.

Applying the results of the paper to other vGIT master spaces will generalize
MSP, as a first step towards studying the structure of higher genus
Gromov-Witten invariants of complete intersections in toric varieties.
It would be interesting to know which of them can produce results like the
BCOV's Feynman rule.
At least for Calabi-Yau complete intersection in product of projective spaces,
BCOV's Feynman rule can be proved \cite{CICYBCOV}.

For Fermat singularities, Huaigong Zhang's work in progress \cite{Huaigong1, Huaigong2}
uses the results of this paper to construct an MSP moduli, obtains its BCOV
Feynman rule as well as its consequences including finite generation, the HAE,
and the analyticity of the partition function.

When applied to target spaces other than vGIT master spaces, it also
produces interesting counting theories. 
For example, in Section~\ref{sec:projective-space-example}, we discuss
projective spaces with various $R$-charges. In this case the target space is already proper so 
the super-potential or cosection localization will not be needed.
A work in progress by Guo-Janda-Lei-Zhang employs the construction
in this work to provide an enumerative geometric construction of the well-known
GUE matrix model. In Section~\ref{sec:complete-intersection-example}, we
discuss a family of LG theories for complete intersections, which
  are related to hybrid models in physics (c.f.\
  \cite{hybrid1,hybrid2,hybrid3,hybrid4,hybrid5}).

\medskip
\noindent
\textbf{Acknowledgments}
We would like to thank Professors Bumsig Kim, Yongbin Ruan, Jun Yu,
Rachel Webb, Dingxin Zhang for their valuable discussion with the authors.

Both
J.~Li and Y.~Zhou are partially supported by the National Key Res.\ and
Develop.\ 
Program of China \#2020YFA0713200.
H.-L.~Chang is
partially supported by NSFC grants 24SC04, the Fundamental Research Funds for
the Central Universities, and Hong Kong grant GRF 16303122.
S.~Guo is partially supported by NSFC 12225101 and 11890660.  J.~Li is also partially
supported by NSFC 12071079, and
by Shanghai SF grants 22YS1400100. W.-P.~Li is partially supported by grants
Hong Kong GRF16304119, GRF16306222 and GRF16305125.
Y.~Zhou is also partially supported by Shanghai Sci.\ and Tech.\ Develop.\ Funds
\#22QA1400800 and
Shanghai Pilot Program for Basic Research-Fudan Univ.\ 21TQ1400100 (22TQ001).
Y.~Zhou would also like to thank the
support of Alibaba Group as a DAMO Academy Young Fellow, and would like to
thank the support of Xiaomi Corporation as a Xiaomi Young Fellow.

\section{A brief outline of the theory}
\label{sec:brief}
We begin with the notion of prestable
LG-quasimaps, introduced in
\cite{fan2017mathematical}. Our setup is slightly more general.
It relies on
what we call
an ($R$-charged) package, consisting of
\begin{equation}
  \label{eq:outlie-package}
  (V,G\le\Gamma, \varpi, \vartheta),
\end{equation}
where $V$ is an affine scheme acted on by
a reductive group $\Gamma$;
$\varpi, \vartheta\in \widehat{\Gamma}$
are characters of $\Gamma$; $G = \ker{(\varpi)}$ and $\varpi$ induces
an isomorphism $\Gamma/G \cong \mathbb C^*$.
The action of this $\mathbb C^*$ (e.g.\ Lemma~\ref{lem:second-decomposition}) is
similar to the $R$-charge in Landau-Ginzburg model, although we do not always
require a super-potential in this paper.

Note that $G$ is also reductive.
We set $\theta = \vartheta|_{G}\in \widehat{G}$, and denote
the $\theta$-stable (resp.~semistable, resp.~unstable) locus by $V^{\mathrm{s}}(\theta)$
(resp.~$V^{\mathrm{ss}}(\theta)$, resp.~$V^{\mathrm{un}}(\theta)$).
We require that $V^{\mathrm{s}}(\theta)=V^{\mathrm{ss}}(\theta)\ne\emptyset$.

Let
\[
  X = [V^{\mathrm{s}}(\theta)/G]
\]
and we call the package $(V,G\le\Gamma, \varpi, \vartheta)$ a package of $X$.
As customary, we use the bracket to denote the quotient stack and use
$V\git_{\theta} G$ to denote the GIT quotient. Thus the quasi-projective scheme
$V\git_{\theta} G$ is the coarse moduli of the separated DM stack $X$.
The role of $X$ is similar to the target space in Gromov-Witten theory
(c.f.\ Remark~\ref{rmk:target-space}).

For a character $\alpha\in \widehat\Gamma$, we associate to it the
$\Gamma$-equivariant line bundle $L_\alpha$ on $V$, whose total space is  $\mathbb C_V\defeq V\times
\mathbb C$,
with $\Gamma$-action given by
$\gamma\cdot(v,a)=(\gamma\cdot v, \alpha(\gamma)a)$ for any $\gamma \in \Gamma$. By abuse of notation, we
also use $L_{\alpha}$ to denote its descent to $[V/\Gamma]$. Similarly, for
$\alpha\in \widehat{G}$, we use $L_{\alpha}$ to denote its associated $G$-equivariant
line bundle on $V$ or its descent to $[V/G]$.

We follow the convention of
\cite[Section~4]{abramovich2002compactifying} on
balanced twisted nodal curves; thus, given a twisted curve $\sC$
with marking $\Sigma^\sC$, a special point means a reduced irreducible
$0$-dimensional substack $z\sub\sC$ that is either a marking or a node of
$\mathscr C$. We denote the log dualizing bundle of $(\mathscr C, \Sigma^{\mathscr
  C})$ by $\omega_{\mathscr C}^{\log} =
\omega_{\mathscr C}(\Sigma^{\mathscr C})$.

\begin{definition} \label{def:prestable-LG-quasimap}
  Given a package of $X$, a genus-$g$ $k$-pointed prestable LG-quasimap to $X$ is
  \[
    \xi=(\mathscr C,\Sigma^\sC, u, \kappa),
  \]
  where
  \begin{enumerate}
  \item $\Sigma^\sC \subset \mathscr C$ is a $k$-pointed genus-$g$ balanced twisted nodal curve;
  \item
    $u: \sC \to [V/\Gamma]$ is a representable morphism, such that the base locus
    $u\upmo([V^{\mathrm{us}}(\theta))/\Gamma]$ of $u$
    is discrete and away from the special points of $(\mathscr C, \Sigma^\sC)$,
  \item
    $\kappa: u^*L_{\varpi} \to \omega^{\log}_\sC$ is an isomorphism of line bundles.
  \end{enumerate}
\end{definition}

Note that the definition above is parallel to Definition~4.2.2 of
\cite{fan2017mathematical}, but we are considering more general target
spaces here, e.g. Section~\ref{sec:projective-space-example}.

\begin{rema}
  \label{rmk:target-space}
  It is easy to see that $V^{\mathrm{ss}}(\theta)$ is $\Gamma$-invariant
  (Corollary~\ref{cor:Gamma-invariance-unstable-locus}) and thus $\varpi$
  induces a $\mathbb C^*$-action on $X$.
  Roughly speaking, an LG-quasimap associated to the package is a quasimap to $X$
  ``twisted'' by $\omega^{\log}_\sC$ and the $\mathbb C^*$-action on $X$.
  Indeed, locally on $\mathscr C$, once we pick a local trivialization for
  $\omega^{\log}_\sC$, the LG-quasimap becomes a map to $X$ defined away from
  the base locus (c.f.\ Lemma~\ref{lem:map-vs-LG}). Moreover, for different
  trivializations $s_1$ and $s_2$, the two maps
  differ by the action of $\frac{s_1}{s_2}$ on $X$.
\end{rema}
\smallskip
We now introduce our $\Omega$-stability of an LG-quasimap.
Here $\Omega$ is a triple
\begin{equation}
  \label{eq:Pi}
  \Omega = (S, A, \vartheta),
\end{equation}
where 
$\vartheta \in \Gamma$ is already in \eqref{eq:outlie-package}, $A \in \mathbb Q$, and $S$ is a finite set of
nonzero  homogeneous elements in the $R_{+}$ to be defined,
subject to condition \eqref{eq:A-greater-than-slope} below.
We first clarify the meaning of homogeneous elements.
For integers $k$ and $c$, let
\[
  R_{k\theta} = H^0(V, L_{k\theta})^{G}, \quad R_{k\vartheta, c\varpi} = H^0(V,
  L_{k\vartheta + c\varpi})^{\Gamma}.
\]
We will see that (Lemma~\ref{lem:second-decomposition})
\begin{equation}
  \label{eq:second-decomposition}
  R_{k\theta} = \bigoplus_{c\in \mathbb Z} R_{k\vartheta, c\varpi}.
\end{equation}
Set
\[
  R = \bigoplus_{k\geq 0} R_{k\theta}, \quad R_{+} = \bigoplus_{k> 0} R_{k\theta}.
\]
Thus $V\git_{\theta} G = \mathrm{Proj}(R)$ and $R_{+}$ generates the
ideal of $V^{\mathrm{us}}(\theta)$.
Note that \eqref{eq:second-decomposition} makes $R$ into a bi-graded ring.
A homogeneous element of $R$ is an element in $R_{k\vartheta,c\varpi}$ for some $(k,c)$.

Let $S\subset R_{+}$ be a finite set of nonzero homogeneous elements.
We say $S$ is \textit{full} when
\begin{equation}
  \label{eq:being-full}
  V^{\mathrm{ss}}(\theta) = \{x \in V \mid \exists f\in S, f(x) \neq 0\},
\end{equation}
i.e.,  when the common vanishing locus of the elements in $S$ equals the unstable
locus $V^{\mathrm{us}}(\theta)$ set-theoretically.

For a homogeneous $f\ne 0\in R_{k\vartheta,c \varpi} \subset R_{k\theta}$,
we define its $\theta$-weight, its $R$-charge and its slope, respectively, to be
\[
  {\mathrm{wt}}_{\theta}(f) = k, \quad {\mathrm{wt}}_{R}(f) = c,
  \quad \mathrm{slope}_{\vartheta}(f) = \frac{c}{k} \ (\mathrm{when}\ k>0).
\]
We will drop the $\vartheta$ from the subscript when it is clear from
the context.

For the stability parameter $A\in \mathbb Q$ in $\Omega$, we require that
\begin{equation}
  \label{eq:A-greater-than-slope}
  A > \max_{f\in
    S}\{\mathrm{slope}(f)\}.
\end{equation}
This completes the definition of the triple $\Omega$.

We call a $\mathbb Q$-line bundle on a twisted curve $\mathscr C$ positive if it has positive
degree on each  irreducible component.
\begin{definition}
  \label{stability}
  A prestable LG-quasimap $\xi$ to $X$
  is $\Omega$-stable if all the following hold:
  \begin{itemize}
  \item[{\bf $\Omega$-1}.]
    $u^{-1}([V(S)/\Gamma])$ is
    discrete and disjoint
    from the special points of $(\sC, \Sigma^\sC )$;
  \item[$\Omega$-{\bf 2}.] for any non-special closed point $x\in \mathscr C$,
    \[
      \min_{f\in S} \{ \frac{1}{\mathrm{wt}_{\theta}(f)} \mathrm{ord}_{x}(u^*f) +
      \mathrm{slope}(f) \}\leq A,
    \]
    where $\mathrm{ord}_x(u^*f)$ denotes the vanishing order of $u^*f$ at $x$;
  \item[$\Omega$-{\bf 3}.] the $\mathbb Q$-line bundle $u^* L_{\vartheta} \otimes
    (\omega_\sC^{\log})^{\otimes A}$
    is positive.
  \end{itemize}
\end{definition}
Note that near any non-special closed point $x\in \mathscr C$, $\mathscr C$ is a
regular scheme of dimension $1$. Thus the vanishing order of a section of a line bundle is
well-defined.

We make a simple observation to simplify $\Omega$-3 in certain situations.
Suppose $\xi$ already satisfies  $\Omega$-2, and $\mathscr C^\prime \subset
\mathscr C$ is an irreducible component.
If $\deg(\omega_{\mathscr C}^{\log}|_{\mathscr C^\prime}) > 0$, then the line
bundle in $\Omega$-3 is automatically positive on $\mathscr C^\prime$; if
$\deg(\omega_{\mathscr C}^{\log}|_{\mathscr C^\prime}) = 0$, then the line
bundle is positive on $\mathscr C^\prime$ unless the LG-quasimap is
a constant regular map on $\mathscr C^\prime$.
See Corollary~\ref{cor:non-negative-components} for the precise statement and proof.

\medskip
We denote the category of $k$-pointed genus-$g$ $\Omega$-stable (resp.~prestable)
LG-quasimaps to $X$ by $\mathrm{LGQ}^{\Omega}_{g,k}(X)$
(resp.~$\mathrm{LGQ}^{\mathrm{pre}}_{g,k}(X)$). For $\xi \in
\mathrm{LGQ}^{\mathrm{pre}}_{g,k}(X)$, we define the
($\vartheta$-)degree of $\xi$ to be $\deg_\sC(u^*L_{\vartheta})$.
We let $\mathrm{LGQ}^{\mathrm{pre}}_{g,k}(X, d)$ and
$\mathrm{LGQ}^{\Omega}_{g,k}(X, d)$ be the open and closed substacks of degree-$d$ prestable and $\Omega$-stable
LG-quasimaps, respectively.

\medskip
The technical part of this paper is to prove the following theorems.

\begin{theo}
  \label{thm:intro-separated-Deligne--Mumford-finite-type}
  The stack $\mathrm{LGQ}^{\Omega}_{g,k}(X,d)$ is a separated Deligne-Mumford stack of finite type.
\end{theo}

\begin{theo}
  \label{thm:intro-properness}
  Suppose $V\git_{\theta}G$ is projective, and suppose
  $S$ is full, then $\mathrm{LGQ}^{\Omega}_{g,k}(X, d)$ is proper.
\end{theo}

As a demonstration, in the remainder of this section we choose appropriate
packages and $\Omega$ to reproduce the stable Mixed-Spin-P(MSP) fields for the
quintic threefold \cite{chang2019mixed},  and deduce the desired geometric properties from the
theorems above.
Other variants like the NMSP fields \cite{chang2018theory} are parallel.
For other complete intersections in weighted projective
spaces, there are obvious generalizations, with some new phenomenon discussed in
Section~\ref{sec:complete-intersection-example}.
However, for complete intersections in toric varieties, there seem to be many
choices for the $R$-charged package, giving
rise to numerous candidates of generalizations of MSP. They all produce
separated moduli spaces with a properly supported virtual cycle.
It will be interesting to know which of them, if
any, will lead to results like the BCOV Feynman sum conjecture. The theorems above
also provide other non-MSP-like new examples of proper moduli spaces with a virtual
cycle, which may be interesting in their own right (e.g.\ Section~\ref{sec:projective-space-example}).

\medskip
To construct MSP for the quintic, we first consider the package $(V_1, G \leq \Gamma, \varpi, \vartheta)$, where
$V_1 = \mathbb C^5 \times \mathbb C \times \mathbb C^2$; $\Gamma = (\mathbb
C^*)^{\times 3}$; $G = (\mathbb C^*)^{\times 2} \leq \Gamma$ consists of the first two factors; $\varpi$ is the
projection onto the last factor; $\vartheta(t_1, t_2, t_3) = t_1t_2^2$ for
$t_1,t_2,t_3 \in \mathbb C^*$; and $\Gamma$ acts by the weight matrix
\[
  \begin{bmatrix}
    1 & 1 &1 &1  & 1 & -5 & 1 & 0\\
    0 & 0 &0 &0 & 0 & 0 &  1 & 1\\
    0 & 0 &0 &0 & 0 & 1 & 0 & 0
  \end{bmatrix}.
\]
Let $x_1 ,\ldots, x_5, p, u, v$ be the standard coordinates on $V_1$. Writing
$\mathbf x = (x_1 ,\ldots, x_5)$, we have
\[
  V_1^{\mathrm{ss}}(\theta) = V_1^{\mathrm{s}}(\theta) = \{(\mathbf x, u) \neq 0 \} \cap \{(u, v) \neq 0 \}
  \cap \{(p, v)\neq 0\}.
\]
Thus the ``target space''
\[
  X_1 = [V_1^{\mathrm{s}}(\theta)/ G],
\]
formed using the first two rows of the matrix above, is a variation of GIT
master space relating $K_{\mathbb P^4}$ and $[\mathbb C^5 / \bm{\mu}_{5}]$
(c.f.\ \cite{thaddeus1996geometric}). And the last row is the so-called
$R$-charge. It accounts for the appearance of the $\omega_{\mathscr C}^{\log}$ below.

For any twisted curve $\mathscr C$, a principal $\Gamma$-bundle $\mathscr P$ on it
induces a triple of line bundles
\[
  (\mathscr L_1, \mathscr L_2, \mathscr L_3), \quad
  \mathscr L_i = \mathscr P \times_{\Gamma} \mathbb C_{\mathrm{pr}_i},
\]
where $\mathbb C_{\mathrm{pr}_i}$ is the copy of $\mathbb C$ on which $\Gamma$ acts
by its projection onto the $i$-th factor.
This is an equivalence between the category of principal
$\Gamma$-bundles on $\mathscr C$ and the category of triples of line
bundles on $\mathscr C$.
Under this correspondence, if $\alpha \in \widehat{\Gamma}$ is given by
\[
  \alpha(t_1, t_2, t_3) = t_1^{a_1}t_2^{a_2}t_3^{a_3}, \quad (t_1, t_2,t_3) \in \Gamma,
\]
and if $\mathscr P$ is the underlying principal
$\Gamma$-bundle of some $u: \mathscr C \to [V_1/ \Gamma]$, then we have a
canonical identification
\[
  u^*L_{\alpha} = \mathscr L_1^{\otimes a_1} \otimes \mathscr L_2^{\otimes a_2}
  \otimes \mathscr L_3^{\otimes a_3}.
\]
In particular, we have $u^*L_{\varpi} = \mathscr L_3$.
Using this, rewriting $\mathscr L = \mathscr L_1$, $\mathscr N =
\mathscr L_2$, a prestable LG-quasimap consists of
\[
  \xi=(\mathscr C,\Sigma^\sC, \mathscr L, \mathscr N, \varphi_1 ,\ldots,
  \varphi_5, \rho, \mu, \nu)
\]
where $\varphi_i \in H^0(\mathscr C, \mathscr L)$, $i=1 ,\ldots, 5$, $\rho \in H^0(\mathscr C,
\mathscr L^{\otimes (-5)} \otimes \omega_{\mathscr C}^{\log})$, $\mu \in
H^0(\mathscr C, \mathscr L \otimes \mathscr N)$, $\nu \in
H^0(\mathscr C, \mathscr N)$, such that
\begin{enumerate}
\item the vector bundle $\mathscr L \oplus \mathscr N$ is representable in the
  sense that the automorphism group of each point of $\mathscr C$ acts
  faithfully on the fiber of $\mathscr L \oplus \mathscr N$;
\item
  the base locus 
  \[
    \{\varphi_1 = \cdots = \varphi_5 = \mu = 0\} \cup \{\mu = \nu = 0 \} \cup
    \{\rho = \nu = 0\}
  \]
  is discrete and away from the special points of $(\mathscr C, \Sigma^{\mathscr
    C})$.
\end{enumerate}
It is easy to see that assuming (2), (1) is equivalent to the condition that
$\mathscr L$ is representable, as in \cite{chang2019mixed}.

We now pick $\Omega = (S, A, \vartheta)$. Let $\vartheta$ be as before,
$\frac{1}{5} < A < \frac{2}{5}$, and let 
\[
  S = \{x_1v^2, \ldots, x_5v^2, uv, u^{10}p\}.
\]
The bi-degrees are, respectively, 
\[
  (1,0) ,\ldots, (1,0), (1,0), (5,1).
\]
Hence \eqref{eq:A-greater-than-slope} is satisfied indeed. 

The $\Omega$-stability reads, in addition to the prestability condition above, that
\begin{enumerate}
\item
  at any smooth point $q \in \mathscr C \setminus \Sigma^{\mathscr C}$, we have
  \begin{align*}
    & \operatorname{ord}_q \varphi +  2 \operatorname{ord}_q \nu\leq  A,\\
    \text{ or }\quad  & \operatorname{ord}_q \mu + \operatorname{ord}_q \nu \leq A,\\
    \text{ or }\quad & \frac{1}{5}\operatorname{ord}_q \rho +
                       2\operatorname{ord}_q \mu \leq A - \frac{1}{5}.
  \end{align*}
\item
  The $\mathbb Q$-line bundle
  \[
    \mathscr L  \otimes  \mathscr N^{\otimes 2} \otimes
    (\omega_{\mathscr C}^{\log})^{\otimes A}
  \]
  is positive.
\end{enumerate}
Since $\frac{1}{5} < A < \frac{2}{5}$ and all the vanishing orders are integers,
(1) is equivalent to that $q$ is not a base point. As is shown in
Appendix~\ref{sec:appendix}, assuming other conditions, (2) is equivalent to
that $\mathrm{Aut}(\xi)$ is finite. Hence we recover the notion of stable MSP
fields in \cite{chang2019mixed}. Then
Theorem~\ref{thm:intro-separated-Deligne--Mumford-finite-type} implies that the
moduli of stable MSP fields of a fixed degree is a separated DM
stack of finite type.

That moduli is in general not proper. In \cite{chang2019mixed}, it is shown that
the cosection degeneracy locus is proper. We define the super-potential to be the
function
\[
  \textstyle
  p(\sum_{i = 1}^5x_i^5): V_1 \longrightarrow \mathbb C,
\]
and denote its critical locus by $V_2$,
namely,
\[
  \textstyle
  V_2 = \{x_1 = \cdots = x_5 = 0\} \cup \{p = \sum_{i=1}^5 x_i^5 = 0\} \subset V_1.
\]
Then the cosection degeneracy locus is equal to
the substack of LG-quasimaps for which $u: \mathscr C \to [V_1/ \Gamma]$ factors
through $[V_2/\Gamma]$. By abuse of notation, we denote the restriction of $S$
to $V_2$ also by $S$. Then $S$ is full and $V_2\git_{\theta} G$ is
projective. Hence, applying Theorem~\ref{thm:intro-properness} to the package
$(V_2, G \leq \Gamma, \varpi, \vartheta)$ with the $\Omega$ above, we recover the
properness of the cosection degeneracy locus in \cite{chang2019mixed}.

\begin{rema}
  It might appear natural to replace $S$ by the set of all homogeneous elements
  in $R_+$, which is essentially the choice in \cite{ciocan2014stable} and
  \cite{fan2017mathematical}. However, if we do that,  the right hand side of
  \eqref{eq:A-greater-than-slope} can be infinity. Indeed, in the 
  example above, the slope of $x_1^{5a+1}p^av^2$ is $a$, for any $a \geq 0$.
\end{rema}

\section{The $\Omega$-stability condition}
\label{sec:moduli-and-stability}
We now study the $\Omega$-stability condition in more details. Following the
notation in the previous section, we will first show that $\Omega$-stability
is an open condition. We then study the change of the stability as we vary
$\Omega = (S, A, \vartheta)$, but keeping $\theta = \vartheta|_{G}$ fixed. We will see that
varying $\vartheta$ and $S$ does not cause essential differences and varying
$A$ results in a family of stability conditions with a wall-and-chamber structure.

\medskip

We first make some simple observations on the ring of invariants.
Let $\mathbb C[V]$ be the ring of regular functions on $V$.
The $\Gamma$ action on $V$ induces an action on $\mathbb C[V]$ by
$\gamma \cdot f = \gamma^*f$, where $\gamma \in \Gamma$ and $f \in \mathbb C[V]$.
Recall that the group of $\theta$-invariants are
\[
  R_{k\theta} = \{f\in \mathbb C[V] \mid g^*f = (\theta(g))^kf,~\forall g\in G
  \}, \quad k \in \mathbb Z_{\geq 0}.
\]
\begin{lemm}
  \label{lem:second-decomposition}
  For any $k\in \mathbb Z_{\geq 0}$, we have $R_{k\theta} = \bigoplus_{c\in \mathbb Z}
  R_{k\vartheta, c\epsilon}$.
\end{lemm}
\begin{proof}
  It is similar to \cite[Proposition~3.2.11]{fan2017mathematical}.
We first claim that $R_{k\theta}$ is $\Gamma$-invariant.
Indeed, for $f\in R_{k\theta}$, $\gamma\in \Gamma$ and $g\in G$,
$$
g^*\gamma^*f=\gamma^* (\gamma g\gamma^{-1})^*f = \gamma^* \theta(\gamma
g\gamma^{-1})^kf  = \gamma^* \theta(
g)^kf =\theta(g)^k\gamma^* f,
$$
using $\vartheta(\gamma g\gamma^{-1})^kf =\vartheta(g)^kf=\theta(g)^kf$, the claim follows.

Now the modified $\Gamma$-action on $R_{k\theta}$
  \[
    (\gamma, f) \mapsto \vartheta(\gamma)^{-k}\gamma^*f
  \]
  is well-defined and
  factors through $\epsilon: \Gamma/G \overset{\cong}{\to} \mathbb C^*$. Thus we can further decompose $R_{k\theta}$
  into a direct sum of eigenspaces with respect to this $\mathbb C^*$-action,
  and the weight-$c$ part is precisely $R_{k\vartheta, c\epsilon}$.
\end{proof}

\begin{coro}
  \label{cor:Gamma-invariance-unstable-locus}
  The unstable locus $V^{\mathrm{us}}(\theta)$ is $\Gamma$-invariant.
\end{coro}

\begin{proof}
  By Lemma~\ref{lem:second-decomposition}, $R_{k\theta}$ is $\Gamma$-invariant
  for each $k$. Hence $R_{+}$, which generates the ideal of $V^{\mathrm{us}}(\theta)$, is also $\Gamma$-invariant.
\end{proof}

Note that for any LG-quasimap $\xi = (\mathscr C, \Sigma^{\mathscr C}, u,
\kappa)$ and any $f\in R_{k\vartheta,c\epsilon}$,
using the $\kappa$ in $\xi$, we have
\begin{equation}
  \label{finL}
  u\sta f\in H^0(\sC, u\sta L_{k\vartheta + c \epsilon}) \cong
  H^0(\sC, u\sta L_{\vartheta}{}^{\otimes k}\otimes(\omega_\sC^{\log})^{\otimes  c}).
\end{equation}

\subsection{The $\Omega$-stability is open}

Recall that $\Omega=(S,A,\vartheta)$ is introduced in
\eqref{eq:Pi} and $A$ satisfies \eqref{eq:A-greater-than-slope}.
\begin{lemm} \label{lem:openness}
  Being $\Omega$-stable is an open condition.
\end{lemm}

\begin{proof}
  It is easy to see that $\Omega$-{1} and $\Omega$-{3} in Definition \ref{stability} are open conditions.
  We argue that assuming $\Omega$-{1} and $\Omega$-{3}, $\Omega$-{2} is also open.

  Let $T$ be a scheme and $\xi=(\sC,\Sigma, u,\kappa)$ be a (flat) $T$-family of prestable LG-quasimaps
  satisfying $\Omega$-1 and $\Omega$-3.
  Then the locus $x\in T$ where $\xi_T|_x$ is stable is constructible. To show it is open, it suffices
  to consider the case where $T = \operatorname{Spec} B$ for $B$ a discrete valuation ring.

  Let $\eta$ and $\zeta$ be the
  generic and closed points of $T$, respectively,
  and let $\xi$ be a prestable LG-quasimaps over $T$. Suppose $\Omega$-1 and
  $\Omega$-3 hold for $\xi$, and
  $\Omega$-{2} holds for its special fiber $\xi_\zeta$.
  We claim $\Omega$-{2} also holds for $\xi_\eta$.

  Suppose not, we can find a closed point $x_\eta$ of $\sC_\eta$
  such that for any $f\in S$ we have
  \begin{equation}
    \label{AA}
    \frac{1}{\mathrm{wt}_{\theta}(f)} \mathrm{ord}_{x_\eta}(u_\eta^*f) +
    \mathrm{slope}(f) > A.
  \end{equation}
  Since $\mathrm{slope}(f) < A$, we have $\mathrm{ord}_{x_\eta}(u_\eta^*f)> 0$,
  so $u^*f|_{x_{\eta}} = 0$ for any $f\in S$.
  Let $x_\zeta$ be a specialization of $x_\eta$.
  Then
  $u^*f|_{x_{\zeta}} = 0$ for any $f\in S$.
  As $\xi_\zeta$ is stable, $x_\zeta$ is non-special
  by $\Omega$-{1}.
  Then by $\Omega$-{2} for $\xi_\zeta$,   there exists some $f\in S$ such that
  \[
    \frac{1}{\mathrm{wt}_{\theta}(f)} \mathrm{ord}_{x_\zeta}(u_\zeta^*f) +
    \mathrm{slope}(f) \leq A.
  \]
  This contradicts to \eqref{AA} and the upper semicontinuity of vanishing orders.
  This proves the lemma.
\end{proof}

\subsection{The dependence on $\Omega$}
In general, different choices of $\Omega$ result in very different moduli spaces.
In simple cases like complete intersections in weighted projective spaces, there
are almost canonical choices of $\Omega$.
However, for more general packages, it is
less obvious which is the best choice. Hence in the remainder of
Section~\ref{sec:moduli-and-stability}, we explore all choices of
$\Omega$ subject to the condition \eqref{eq:A-greater-than-slope} and study the
dependence of the stability on $\Omega$, when $X$ is kept unchanged.
How the theory changes when $X$ changes by variation of GIT is a much harder
problem, which is beyond the scope of this paper.
Hence, in the discussion below, we always fix
\[
  (V, G \leq \Gamma, \epsilon) \quad \text{and} \quad \theta = \vartheta|_{G}.
\]

In summary,  choosing a different $\vartheta$  gives essentially the same set of
stability conditions (Lemma~\ref{lem:dependence-on-C-and-vartheta}). Varying $A$
gives a family of stability conditions with a
wall-and-chamber structure (Corollary~\ref{interval}), similar to the case of
$\epsilon$-stable quasimaps \cite{ciocan2014stable}. We also describe the
stability as $A \to \infty$ in
Corollary~\ref{cor:stability-condition-A-infinity}. It recovers the $\epsilon =
0^+$ stability in
\cite{fan2017mathematical}, which is the only stability condition in
\cite{fan2017mathematical} when there is no so-called ``good lift''.

When $V\git_{\theta} G$ is projective, the stability is
essentially independent of $S$. More precisely, assuming $S$ is full (c.f.\
\eqref{eq:being-full}) in the sense that it contains enough
elements to cut out the unstable locus $V^{\mathrm{us}}(\theta)$
set-theoretically, then by Corollary~\ref{cor:projective-case-independent-of-S},
\begin{enumerate}
\item
  for fixed $A$, the $\Omega$-stability is independent of the choice of $S$;
\item
  if in addition $V\git_{\theta} G$ is reduced, the right hand side of
  \eqref{eq:A-greater-than-slope} is independent of $S$,
  so that we have the same set of allowed $A$.
\end{enumerate}

The dependence on $S$ is more complicated when $V\git_{\theta} G$ is not
projective. However, in most applications, we only care about the stability for
projective and reduced $V\git_{\theta} G$.
Indeed, to define Gromov-Witten type invariants,
either the original GIT quotient is projective or we take $V \git_{\theta} G$ to
be the critical locus of some function, which is assumed
to be projective.
For example, in the case of MSP fields for
the quintic discussed in Section~\ref{sec:brief}, what affects the MSP invariants
is only the stability condition in a neighborhood of the cosection degeneracy locus.
Since the stability is an open condition, all that matters is the stability
for MSP fields in the degeneracy locus.

\medskip
We now begin to prove the results stated above.
It is easy to see how the $\Omega$-stability depends on the choice of $\vartheta$
when $\theta=\vartheta|_G$ is fixed.
In this case two different choices of $\vartheta$ differ by a multiple of
$\epsilon$, and the moduli spaces are related as follows.

\begin{lemm} \label{lem:dependence-on-C-and-vartheta}
  For any $a\in \mathbb Z$, let
  \[
    \Omega=(S,A,\vartheta) \quad \text{and} \quad  \Omega'=(S,A+a,\vartheta- a\epsilon).
  \]
  Then the identity map induces a canonical isomorphism of stacks
  \[
    \mathrm{LGQ}^{\Omega}_{g,k}(X, d) = \mathrm{LGQ}^{\Omega'}_{g,k}(X, d - a(2g-2+k)).
  \]
\end{lemm}

\begin{proof} This is straightforward and will be omitted.
\end{proof}

We next investigate how varying $A$ affects the $\Omega$-stability condition.
Consider any prestable LG-quasimap
\[
  \xi=(\mathscr C,\Sigma^\sC, u, \kappa) \in \mathrm{LGQ}^{\mathrm{pre}}_{g,k}(X,
  d)(\mathbb C). 
\]
Let $\Irre(\xi)$ be the set of irreducible components of $\sC$.
And let $\Delta_{+}$ (resp.\ $\Delta_-$, resp.\
$\Delta_0$) denote the subset of $\Irre(\xi)$ consisting of $\sB$ such that
$\deg(\omega^{\log}_{\mathscr C}|_{\sB})
> 0$, (resp.\ $<0$, resp.\ $=0$).
Set
\[
  s_{\mathrm{max}} = \max_{f\in S}\{\mathrm{slope}(f)\}
\]
and recall that we always require $A > s_{\mathrm{max}}$.

For $z\in V$ we let $\text{Stab}_G(z)$ be the $G$-stabilizer of $z$.
As $G$ is of finite type and $V^{\mathrm{ss}}(\theta) = V^{\mathrm{s}}(\theta)$,
there exists an $N_0$ such that
\[
  \# (\text{Stab}_G(z)) \leq N_0, \quad \forall z \in V^{\mathrm{ss}}(\theta).
\]
\begin{lemm} \label{lem:bounding-the-order-of-isotropy}
 For any orbifold point $x$ of $\mathscr C$, 
 $\#(\mathrm{Aut}(x)) \leq N_0$.
\end{lemm}

\begin{proof}
  Because the $u:\sC\to [V/\Gamma]$ of $\xi$ is representable, $\mathrm{Aut}(x)$ injects into
  $\mathrm{Aut}(u(x))$.
  Because $\Aut(x)$ acts on $\omega_{\mathscr C}^{\log}|_x$ trivially,
  using  the isomorphism $\kappa$, we see that $\mathrm{Aut}(x)$
  injects into $\text{Stab}_G(z)$ where $z\in V$ is any point lying over $u(x)$.
  By the (pre)stability condition, $z\in V^{\mathrm{ss}}(\theta)$. The
  lemma then follows.
\end{proof}

\begin{coro} \label{cor:deg-of-L-takes-discrete-values}
  There exists $D\in \ZZ_{>0}$ such that for any prestable LG-quasimap $\xi$ and any
  $\sB\in\Irre(\xi)$,
  $ \deg(u^* L_{\vartheta}|_{\sB}) \in \frac{1}{D} \mathbb Z$.
\end{coro}
\begin{proof}
  The proof is straightforward, and will be omitted.
\end{proof}

\begin{lemm}
  \label{lem:lower-bound-on-degree-of-line-bundles}
  Suppose $\xi$ satisfies $\Omega$-2.
\begin{enumerate}[(i)]
\item For any $\sB \in \Delta_{+} \cup \Delta_0$,
  $\deg (u^*L_{\vartheta}|_{\mathscr B}) + s_{\mathrm{max}} \deg
  (\omega_{\mathscr C}^{\log}|_{\mathscr B}) \geq 0$.
\item There exists $\delta > 0$, independent of $\xi$ and $A$, such that for any $\sB \in \Delta_{-}$,
  $ \deg (u^*L_{\vartheta}|_{\sB}) + (s_{\mathrm{max}} + \delta)
  \deg (\omega_{\mathscr C}^{\log}|_{\sB}) \geq 0$.
\end{enumerate}
\end{lemm}
\begin{proof}
  First suppose $\sB \in \Delta_{+} \cup \Delta_0$ and
  \begin{equation}
    \label{eq:deg-s-max}
    \deg (u^*L_{\vartheta}|_{\sB}) + s_{\mathrm{max}} \deg (\omega_{\mathscr C}^{\log}|_{ \sB})
  \end{equation}
  is negative.
  Then for any $f\in S \cap R_{k\vartheta,c\epsilon}$, we have
  $s_{\mathrm{max}} \geq \frac{c}{k}$.  Thus
  \[
    \deg (u^*L_{k\vartheta+c\epsilon}|_{\sB}) =
    k \big( \deg (u^*L_{\vartheta}|_{\sB}) + \frac{c}{k}\deg (\omega_{\mathscr C}^{\log}|_{ \sB})\big)<0.
  \]
  Since $u\sta f|_\sB$ is a section of  $u^*L_{k\vartheta + c\epsilon}|_{\sB}$,
  $u\sta f|_\sB=0$.
  This is true for each $f\in S$,
  violating  $\Omega$-2 since the vanishing
  orders are infinity. This proves (i).

  For (ii), by the $\Omega$-stability, for any $\sB\in \Delta_{-}$,
  \[
    \deg (u^*L_{\vartheta}|_{\sB}) + A \deg (\omega_{\sC}^{\log}|_{\sB}) > 0.
  \]
  Using $A
  > s_{\max}$ and $\deg (\omega_{\mathscr C}^{\log}|_{\sB}) < 0$,
  we see
  \eqref{eq:deg-s-max} is positive.
  Applying Corollary~\ref{cor:deg-of-L-takes-discrete-values}  we see that
  \eqref{eq:deg-s-max} has bounded
  denominator.
  Since $\deg (\omega_{\mathscr C}^{\log}|_{\sB}) = -1$ or $-2$, this proves (ii).
\end{proof}
We use the above lemma to simplify the $\Omega$-stability for $\mathscr B \in
\mathrm{Irre}(\xi)$ with $\deg(\omega^{\log}_{\mathscr C}|_{\mathscr B}) \geq 0$.
Note that when $\deg(\omega^{\log}_{\mathscr C}|_{\mathscr B}) = 0$, we have
$\omega^{\log}_{\mathscr C}|_{\mathscr B} \cong \mathscr O_{\mathscr
  B}$. Picking any trivialization of it, the LG-quasimap induces a quasimap
$\mathscr B 
\to  [V/G]$ (c.f.\ Lemma~\ref{lem:map-vs-LG}).

\begin{coro}
  \label{cor:non-negative-components}
  Suppose $\xi$ satisfies $\Omega$-2.
  For any $\mathscr B \in \mathrm{Irre}(\xi)$,
  \begin{enumerate}
  \item
    if $\deg(\omega^{\log}_{\mathscr C}|_{\mathscr B}) > 0$, then the
    line bundle in $\Omega$-3 is automatically positive on $\mathscr B$;
  \item
    if $\deg(\omega^{\log}_{\mathscr C}|_{\mathscr B}) = 0$, then the
    line bundle in $\Omega$-3 is positive on $\mathscr B$ if and only if the
    induced quasimap $\mathscr B \to [V/G]$ is a constant regular map, i.e.,\ it
    maps $\mathscr B$ into $X$ and the induced map $\mathscr B \to
    V\git_{\theta}G$ is constant.
  \end{enumerate}
\end{coro}
\begin{proof}
  The proof follows immediately from (i) above and the requirement that $A >
  s_{\max}$ (c.f.\ \cite[Lemma~2.3]{cheong2015orbifold}).
\end{proof}

\begin{coro}\label{interval}
  Having fixed the target $(V,G\le\Gamma, \epsilon, \vartheta)$ and $S$, 
  there exists an integer $M > 0$, independent of $g$ and $k$, such that
  the $\Omega$-stability condition does not change as $A$ varies in each interval
  \[
    \big[ \frac{i}{M}, \frac{i+1}{M} \big) \cap (s_{\max}, +\infty), \quad
    \forall i \in \mathbb Z.
  \]
\end{coro}
\begin{proof}
  This follows immediately from
  Corollaries~\ref{cor:deg-of-L-takes-discrete-values} and \ref{cor:non-negative-components}.
\end{proof}
\begin{rema}
  If $\theta$ decendes to the coarse moduli of $X$, $S$ only consists of
  weight $1$ elements, and $(g, k) \neq (0, 0)$, then we can take $M = 1$ above.
  In case $(g, k) = (0, 0)$, we may have an additional wall at $A = \frac{d}{2}$.
\end{rema}

We now study the stability condition as $A\to \infty$.
\begin{lemm}
  Suppose $\xi$ is stable.
  There are constants $C_1$ and $C_2$ depending only on $S,\vartheta, g, k, d$ but
  independent of
  $\xi$ and $A > s_{\mathrm{max}}$, such that
\begin{enumerate}
  \item $\#\Irre(\xi)\le C_1$, and
  \item $|\deg (u^*L_{\vartheta}|_\sB)| \le C_2$, for any $\sB\in\Irre(\xi)$.
\end{enumerate}
\end{lemm}

\begin{proof}
  Putting (i) and (ii) of Lemma~\ref{lem:lower-bound-on-degree-of-line-bundles} together, we have
  \[
    0 \leq d + \sum_{\sB\in \Delta_+} s_{\mathrm{max}} \deg
    (\omega_{\mathscr C}^{\log}|_{\sB}) + \sum_{\sB \in
      \Delta_{-}}(s_{\max} + \delta) \deg (\omega_{\mathscr
      C}^{\log}|_{\sB}).
  \]
  This implies
  \begin{equation} \label{eq:bound-on-negative-omega}
    \sum_{\sB\in \Delta_{-}}
    \deg (\omega_{\mathscr C}^{\log}|_{\sB}) \geq
    \frac{s_{\max}}{\delta}(2 - 2g - k) - \frac{d}{\delta} ;
  \end{equation}
  and
  \begin{equation} \label{eq:bound-on-positive-omega}
    \sum_{\sB \in\Delta_{+}} \deg (\omega_{\mathscr C}^{\log}|_{\sB})
    \le 2g - 2 + k  + \frac{s_{\max}}{\delta}(2g - 2 + k) + \frac{d}{\delta}.
  \end{equation}
  As $\deg(\omega_{\mathscr C}^{\log}|_{\sB})\in\ZZ$ for $\sB \in \Irre(\xi)$, both
  $\#\Delta_+$ and $\#\Delta_{-}$ are bounded.

  We decompose $d = \deg (u^*L_{\vartheta})$ into three parts
  \[
    d = \sum_{\mathscr B \in \Delta_+} \deg (u^*L_{\vartheta}|_\sB) + 
    \sum_{\mathscr B \in \Delta_-} \deg (u^*L_{\vartheta}|_\sB) + 
    \sum_{\mathscr B \in \Delta_0} \deg (u^*L_{\vartheta}|_\sB).
  \]
  By \eqref{eq:bound-on-positive-omega}, each $\deg(\omega_{\mathscr
    C}^{\log}|_{\sB})$ is bounded from above. Hence,
  by Lemma~\ref{lem:lower-bound-on-degree-of-line-bundles}, $\deg(u^*L_{\vartheta}|_{\sB})$ is
  bounded from below.  By $\Omega$-3 and
  Corollary~\ref{cor:deg-of-L-takes-discrete-values}, we have
  $\deg(u^* L_{\vartheta}|_{\mathscr B}) > \frac{1}{D} > 0$ for each $\mathscr B \in
  \Delta_0$. Then it is easy to see that $\#\Delta_{0}$ is also bounded and
  each summand must be also bounded from above.
\end{proof}

\begin{coro} \label{cor:stability-condition-A-infinity}
  Fixing $S, g, k, d$, for any $\vartheta$ and $A \gg 0$,
  assuming $\Omega$-1, then $\Omega$-2 is automatic and $\Omega$-3 is equivalent to
  \begin{enumerate}
  \item[$\Omega$-\textbf{3'}.] for any $\sB\in\Irre(\xi)$,
    either $\deg(\omega^{\log}_{\mathscr C}|_{\sB}) > 0$, or
    $\deg (u^*L_{\vartheta}|_{\sB}) > \deg(\omega^{\log}_{\mathscr
      C}|_{\sB}) = 0$.
  \end{enumerate}
  In particular, the $\Omega$-stability is independent of $A\gg 0$ and $\vartheta$.
\end{coro}
\begin{proof}
  To see that it is independent of $\vartheta$, note that two different choices
  differ by a multiple of $\epsilon$. Using $u^*L_{\epsilon} \cong \omega_{\mathscr
    C}^{\log}$, we have $u^*L_{\vartheta + a\epsilon} = u^*L_{\vartheta} \otimes
  (\omega_{\mathscr C}^{\log})^{\otimes a}$ for any $a\in \mathbb Z$. Thus when
  $\deg(\omega_{\mathscr C}^{\log}|_{\sB} )= 0$,
  $\deg(u^*L_{\vartheta}|_{\sB}) = \deg(u^*L_{\vartheta +
    a\epsilon}|_{\sB})$. Hence the $\Omega$-3' above does not depend on the
  choice of $\vartheta$. The rest is obvious.
\end{proof}

\subsection{Dependence on $\Omega$ in the projective case}
\label{sec:stability-proper-case}


Recall that $R_+=\bigoplus_{k\ge 1 , c} R_{k\vartheta,c\epsilon}$;
it generates the ideal defining the unstable locus
$V^{\mathrm{us}}(\theta)$.
%
%
Recall that $S\subset R_+$ contains only nonzero homogeneous elements in $R_+$,
and $S$ is said to be full if
\begin{equation}
  \label{eq:f_1_to_r_defines_unstable_locus}
  V^{\mathrm{ss}}(\theta) = \{x \in V \mid \exists f\in S, f(x) \neq 0\}.
\end{equation}
As always,
we assume
\begin{equation}
  \label{eq:A-big-in-projective-subsection}
  A > \max\{\mathrm{slope}(f) \mid f\in S\}.
\end{equation}


We now discuss the dependence of the stability condition on the choices of
$S,A,\vartheta$ in the case where $V\git_{\theta}
G$ is
projective and $S$ is full, assuming $\theta = \vartheta|_{G}$ is fixed.
  The dependence on $\vartheta$ is already given by Lemma~\ref{lem:dependence-on-C-and-vartheta}.
  Moreover, 
  it is easy to see that the decomposition
  $R_{0} = \bigoplus_{c} R_{0, c\epsilon}$
is independent on the choice of $\vartheta$,
where $R_{0} = \mathbb C[V]^{G}$.

In the following we will further assume that $R_0 = R_{0,0}$. In the
projective case this will be
satisfied if we assume $V$ is reduced, which is enough for most of our
applications.



\begin{lemm}\label{lem:C-max}
Suppose $R_{0} = R_{0, 0}$ and $S$ is full, then
\[
  \max_{f\in S}\{\mathrm{slope}(f)\} =
  \max \{\mathrm{slope}(f) \mid
  0 \neq  f \in R_+, \text{$f$ is homogeneous} \}
\]
is independent of the choice of $S$.
\end{lemm}

\begin{proof}
We first claim that
\begin{equation}
    \label{eq:set-of-all-slopes}
    \{\mathrm{slope}(f)  \mid 0 \neq  f \in R_+, \text{$f$ is homogeneous}\}
\end{equation}
has a maximum.
Observe that for any homogeneous $f,g\in R_+$ satisfying $fg\neq 0$, we have
\begin{equation} \label{eq:slope-inequality}
    \mathrm{slope}(fg) \leq \max\{\mathrm{slope}(f), \mathrm{slope}(g)\}.
\end{equation}
Moreover, the inequality is strict unless $\slope(f)=\slope(g)$.
Since $R_{\bullet}$ is a finitely
generated $R_{0}$-algebra and we have assumed $R_0 = R_{0,0}$,
\eqref{eq:set-of-all-slopes} has a
maximum, achieved at a homogeneous
generator  $h \in R_{+}$.

Let $S$ be full, it suffices to show that
  \begin{equation}
    \label{eq:slope-h-smaller-than-max-in-S}
    \mathrm{slope}(h) \leq \max_{f\in S}\{\mathrm{slope}(f)\}.
\end{equation}
Combined with that $S\subset R_+$, this will prove the lemma.

By \eqref{eq:f_1_to_r_defines_unstable_locus} and Hilbert's Nullstellensatz, $h$
is contained in the radical
of the ideal generated by $S$. 
Thus there are regular functions $g_i$ on $V$, and $f_i \in S$ so that
\begin{equation} \label{eq:h-as-a-comb-of-f_i}
    h^m = g_1 f_1 + \cdots  + g_r f_r.
\end{equation}




Suppose that $h \in R_{k\theta, c\epsilon}$ and $f_i\in S\cap R_{k_i\theta,
  c_i\epsilon}$, $i=1 ,\ldots, r$.
By a standard argument using the Reynolds operator, we can assume
$g_i \in R_{(mk-k_i)\theta, (mc-c_i)\epsilon}$ for each $i$. Thus in particular
$\mathrm{slope}(g_if_i) = \mathrm{slope}(h)  = c/k$ unless $g_if_i = 0$.

Pick any $g_{i_0}f_{i_0}\neq 0$. We claim that $\mathrm{slope}(f_{i_0}) =
\mathrm{slope}(h)$. Indeed,
first consider the case $k_{i_0} < mk$.  Thus $g_{i_0}\in R_+$.
By the maximality assumption on $\slope(h)$, \eqref{eq:slope-inequality} for
$f_{i_0}$ and $g_{i_0}$
cannot be strict.
By the statement immediately after \eqref{eq:slope-inequality}, we must have
$\slope(g_{i_0})=\slope(f_{i_0})=\slope(g_{i_0}f_{i_0}) = \slope(h)$. Next consider the
case $k_{i_0} = mk$. Thus $g_{i_0} \in R_{0} = R_{0, 0}$,
by the assumption of the lemma. Thus $c_{i_0} = mc$ and $\slope(f_{i_0}) =
\slope(h) = c/k$. Thus we have proved \eqref{eq:slope-h-smaller-than-max-in-S}
and the lemma follows.
\end{proof}

  Let $\xi = (\mathscr C,\Sigma^{\mathscr C}, u, \kappa)$
  be a prestable LG-quasimap to $X$ and $x\in \mathscr C$ be a nonspecial closed
  point.
  For any homogeneous $f\neq 0 \in R_{+}$, we abbreviate
  \[
    \overline{\mathrm{ord}}_x(f) = \frac{1}{\mathrm{wt}_{\theta}(f)}
    \mathrm{ord}_{x}(u^*f) + \mathrm{slope}(f),
  \]
  which is a quantity used in $\Omega$-2.
  \begin{lemm}
    \label{lem:min-order}
    Suppose $R_{0} = R_{0, 0}$ and $S$ is full, then
    \[
      \min_{f\in S}\{ \overline{\mathrm{ord}}_x(f)\} =
      \min \{ \overline{\mathrm{ord}}_x(f) \mid f \in R_+,
      \text{$f\ne 0$ homogeneous}\}
    \]
    is independent of the choice of $S$.
  \end{lemm}
\begin{proof}
  It is parallel to Lemma~\ref{lem:C-max}. We
  sketch the proof while indicating the difference.
  First note that $\overline{\mathrm{ord}}_x(f) = \infty$ if and only if $u^*f$ is
  identically zero near $x$. Thus by the prestability condition,
  $\overline{\mathrm{ord}}_x(f)<\infty$ for some $f\in R_{+}$.

  It is easy to show that for any homogeneous elements $f,g\in R_{+}$ such
  that $fg\neq 0$ and $\overline{\mathrm{ord}}_{x}(fg)< \infty$, we have
  \begin{equation}
    \label{eq:reduced-order-inequality}
    \overline{\mathrm{ord}}_x(fg) \geq \min\{\overline{\mathrm{ord}}_x(f),
    \overline{\mathrm{ord}}_x(g)\},
  \end{equation}
  which is strict unless $\overline{\mathrm{ord}}_x(f) =
  \overline{\mathrm{ord}}_x(g)$. Further, if $f\in R_+$ and $g\in R_{0} =
  R_{0,0}$ are homogeneous with $fg \neq 0$, then
  we have
  \[
    \overline{\mathrm{ord}}_x(fg) \geq \overline{\mathrm{ord}}_x(f).
  \]
  Using that, as in the proof of Lemma~\ref{lem:C-max}, the minimum
  \[
    \min\{\overline{\mathrm{ord}}_x(f) \mid 0\neq f\in R_+, f\text{ is homogeneous}\} < \infty
  \]
  is achieved at some homogeneous $h\in R_+$. 
  As before, write
  \[
    h^m = g_1 f_1 + \cdots  + g_r f_r.
  \]
  It is easy to see that $\overline{\mathrm{ord}}_x(g_{i_0}f_{i_0}) \leq
  \overline{\mathrm{ord}}_x(h^m) = \overline{\mathrm{ord}}_x(h)$ for some $i_0$,
  by the superadditivity of $\mathrm{ord}_{x}(\cdot)$.
  By the minimality of $\overline{\mathrm{ord}}_x(h)$, we must actually have
  $\overline{\mathrm{ord}}_x(g_{i_0}f_{i_0}) = \overline{\mathrm{ord}}_x(h)$.
  Then as in the proof of Lemma~\ref{lem:C-max}, one uses
  \eqref{eq:reduced-order-inequality} to show that $\overline{\mathrm{ord}}_x(h)
  = \overline{\mathrm{ord}}_x(f_{i_0})$. Indeed, the $k_{i_0}<mk$ case is
  verbatim; if $k_{i_0} = mk$, then $g_{i_0}\in R_{0} =
  R_{0,0}$ and we have $ \overline{\mathrm{ord}}_x(f_{i_0}) \leq
  \overline{\mathrm{ord}}_x(g_{i_0}f_{i_0}) = \overline{\mathrm{ord}}_x(h)$. Again, by the
  minimality of $\overline{\mathrm{ord}}_x(h)$, we must have
  $\overline{\mathrm{ord}}_x(f_{i_0}) = \overline{\mathrm{ord}}_x(h)$.
  This completes the proof.
\end{proof}
Summarizing,
\begin{coro}
  \label{cor:projective-case-independent-of-S}
  Given $\vartheta$, suppose $V$ is reduced, $V\git_{\theta} G$ is projective,
  and $S$ is full. Then the right hand side of \eqref{eq:A-big-in-projective-subsection} is
  independent of the choice of $S$. Further, once $A$ is chosen, the $\Omega$-stability does
  not depend on the choice of $S$.
\end{coro}
\begin{proof}
  First note that since $S$ is full, $\Omega$-1 is already contained in the
  prestability condition.
  We claim that $R_{0} = R_{0, 0}$.
  Indeed, since $V\git_{\theta} G$ is projective,
  $R_{0} = \mathbb C[V]^{G}$ is a finite $\mathbb C$-algebra.
  Suppose $f\in R_{0,c\epsilon}$ for $c\neq 0$, then for sufficiently
  large $n$, $f^n\in R_{0,nc\epsilon} = 0$. Since $V$ is reduced and
  $R_{0}\subset \mathbb C[V]$, we have $f=0$.
  Now we have proved the claim. Applying the previous two lemmas finishes the proof.
\end{proof}

\section{Examples}
\label{sec:examples}
\subsection{Projective space with various $R$-charges}
\label{sec:projective-space-example}

This example is crucial to the proof of the main theorems in this paper.

We first fix an $R$-charged package for $\PP^{N-1}$.
Let $\Gamma = (\mathbb C^*)^{\times 2}$ and let
$G = \mathbb C^* \le \Gamma$ be the first factor;
let $V = \CC^N$ on which $\Gamma$ acts by weights
\[
  \begin{bmatrix}
    1 & \cdots & 1 \\
    c_1 & \cdots & c_N
  \end{bmatrix}.
\]
Let $\theta: G = \mathbb C^* \to \mathbb C^*$ be the identity map. Thus
$X = V\git_{\theta} G = \mathbb P^{N-1}$.
Let $\vartheta, \epsilon:\Gamma\to \CC\sta$ be the
projections onto the first and second factors, respectively.

\medskip
As in Section~\ref{sec:introduction}, principal $\mathbb C^*$-bundles are
equivalent to line bundles.
This way a prestable LG-quasimap to $X$ consists of
\[
  (\mathscr C, \Sigma^{\mathscr C}, \mathscr L, \varphi),
\]
where
\begin{itemize}
\item $(\mathscr C, \Sigma^{\mathscr C})$ is a pointed balanced twisted curve;
\item $\mathscr L$ is a representable line bundle on $\mathscr C$;
\item
  $\varphi = (\varphi_1 ,\ldots, \varphi_N) \in \Gamma(\mathscr C, \bigoplus_{i
    = 1}^N\mathscr L \otimes (\omega_{\mathscr C}^{\log})^{\otimes c_i})$, such
  that $\{\varphi = 0\}$ is finite and disjoint from the special
  points of $(\mathscr C, \Sigma^{\mathscr C})$.
\end{itemize}

Since $X$ is a scheme, by the representability assumption in
Definition~\ref{def:prestable-LG-quasimap},
$\mathscr C$ must also be a scheme (c.f.\ Lemma~\ref{lem:bounding-the-order-of-isotropy}).
Now let $x_1 ,\ldots, x_N$ be the standard coordinates on $\CC^{N}$.
Each $x_i$ has bi-degree $(1, c_{i})$;
$\mathrm{slope}(x_i) = c_i$. Let $S = \{x_1 ,\ldots, x_N\}$; fix
$A > s_{\mathrm{max}} := \max\{c_i\mid i = 1 ,\ldots, N\}$, and let $\Omega=(S,A,\vartheta)$.

Let $\mathrm{LGQ}^{\Omega}_{g,k}(\mathbb P^{N-1},d)$ be the moduli of
degree-$d$, $\Omega$-stable LG-quasimaps to $\mathbb
P^{N-1}$, where $d$ is by definition the degree of $\mathscr L$.
By definition, it parametrizes prestable LG-quasimaps
\begin{equation}
  \label{eq:LG-quasimap-to-P-N-1-in-properness-section}
  \xi = (\mathscr C, \Sigma^{\mathscr C}, \mathscr L, \varphi)
\end{equation}
that satisfies the following $\Omega$-stability condition (Definition~\ref{stability}):

\smallskip
\noindent
{\bf $\Omega$-stability for $\PP^{N-1}$.}{\sl
  \label{cnd:stability-P-N-1}
  For a prestable LG-quasimap $\xi$ to $\PP^{N-1}$,
  \begin{itemize}
  \item[$\Omega$-\textbf{2}.] for non-special $x\in \mathscr C$,
    $\min_{i=1}^N\{ \mathrm{ord}_x (\varphi_i) + c_i\} \leq A$,
  \item[$\Omega$-\textbf{3}.] the $\mathbb Q$-line bundle $\mathscr L \otimes
    (\omega_\sC^{\log})^{\otimes A}$ is positive.
  \end{itemize}
  Note that in this case $\Omega$-{1} of Definition~\ref{stability} follows from
  the prestability condition.
}
It carries a natural virtual fundamental class
constructed the same way as
in \cite{chang2012gromov}.
Since $\mathbb P^{N-1}$ is projective and $S$ is full, the moduli of $\Omega$-stable LG-quasimaps
of a fixed degree is proper. No cosection localization procedure is needed in this case.

\subsection{Complete intersections in projective spaces}
\label{sec:complete-intersection-example}
We consider a smooth complete intersection $Y$ of $s$ hypersurfaces of degrees $\ell_1
,\ldots, \ell_s$ in $\mathbb P^{N-1}$. Parallel to the quintic case,
the complete intersection can be realized as the critical locus of the function
\[
  \textstyle
  \sum_{i = 1}^s p_i F_i: X_+ = \mathrm{Tot}(\mathcal O_{\mathbb P^{N-1}}(-\ell_1) \oplus \cdots \oplus
  \mathcal O_{\mathbb P^{N-1}}
  (-\ell_s)) \longrightarrow \mathbb C,
\]
where $F_1 ,\ldots, F_s$ are the defining equations for $Y$ and $p_i$ is the
standard coordinate on the fibers of $\mathcal O_{\mathbb P^{N-1}}(-\ell_{i})$.
For an appropriately chosen $R$-charged package, writing $X_{+}$ as a GIT
quotient of $\mathbb C^{N+s}$ by $\mathbb C^*$, the theory of LG-quasimaps is a formal
generalization of the $p$-field construction in \cite{chang2012gromov}, and is
mentioned in \cite[\S7.2]{fan2017mathematical}. The
invariants are equivalent to the Gromov-Witten invariants or quasimap invariants
of $Y$ (c.f.\
\cite{chang2020invariants,chen2021fields}).
When $Y$ is Calabi-Yau (CY), this is the
so-called CY phase of the theory.

Varying the polarization for the GIT, we get another quotient
\[
  X_{-} = \mathrm{Tot}(\mathcal O_{\mathbb P(\ell_1, ,\ldots,
    \ell_{N})}(-1)^{\oplus N}),
\]
giving rise to the Landau-Ginzburg (LG) phase, generalizing Witten's $r$-spin
classes
\cite{witten1993matrix, chang2015witten}
and the Fan-Jarvis-Ruan-Witten
theory \cite{fan2013witten}. It was called hybrid Landau-Ginzburg/sigma model in
\cite{witten1993phases}.
When $\ell_1 = \cdots = \ell_s$, it was studied in
\cite{clader2017landau}. We consider the LG phase
for general $\ell_1 ,\ldots, \ell_s$.

Take $\Gamma = \mathbb C^*\times \mathbb C^*$ acting on $V = \mathbb C^N \times
\mathbb C^s$ by weights
\[
  \begin{bmatrix}
    1 & \cdots & 1 & -\ell_1 & \cdots & - \ell_s\\
    0 & \cdots & 0 & 1 & \cdots & 1
  \end{bmatrix}.
\]
Take $\epsilon: \Gamma \to \mathbb C^*$ to be the projection onto the
second factor, and thus $G = \ker(\epsilon)$ is the first
factor subgroup of $\Gamma$. Take $\vartheta:\Gamma\to \mathbb C^*$ to be the
inverse of the projection onto the first factor, and $\theta = \vartheta|_{G}$.

Let $x_1 ,\ldots, x_N, p_1 ,\ldots, p_s$ be the standard coordinates on $\mathbb
C^{N}\times \mathbb C^{s}$.
Thus the unstable locus is defined by $p_1 = \cdots = p_s = 0$, and $X =
[V^{\mathrm{s}}(\theta) / G]$ is the $X_{-}$ above.

Using the standard correspondence between line bundles and principal $\mathbb
C^*$-bundles as before, we can describe the moduli space as follows.
A prestable LG-quasimap to $X$ consists of
\[
  (\mathscr C, \Sigma^{\mathscr C}, \mathscr L, \varphi_1 ,\ldots, \varphi_N,
  \rho_1 ,\ldots, \rho_s),
\]
where
\begin{itemize}
\item $(\mathscr C, \Sigma^{\mathscr C})$ is a pointed twisted curve with
  balanced nodes;
\item $\mathscr L$ is a representable line bundle on $\mathscr C$;
\item $\varphi_i \in \Gamma(\mathscr C, \mathscr L)$, $i=1 ,\ldots, N$;
\item $\rho_i \in \Gamma(\mathscr C, \mathscr L^{\otimes (-\ell_i)} \otimes
  \omega_{\mathscr C}^{\log})$, $i=1 ,\ldots, s$, such that the common zeros of
  $(\rho_1 ,\ldots, \rho_s)$ are finite and disjoint from the nodes and
  markings.
\end{itemize}

Now we come to the $\Omega$-stability condition in Definition~\ref{stability}.
Note that the function $p_i$ has bi-degree $(\ell_i, 1)$, $i=1 ,\ldots, s$.
Let $\ell_{\max}$ and
$\ell_{\mathrm{min}}$ be the maximum and minimum of $\ell_1 ,\ldots, \ell_s$,
respectively. Take
\[
  S = \{p_1 ,\ldots, p_s\}.
\]
Then
\[
  \max\{\mathrm{slope}(f) \mid f \in S\} = \frac{1}{\ell_{\mathrm{min}}}.
\]
\ifdefined\SHOWOLDSTABILITY {

  TO BE OMITTED

  Used to be:
  \[
    S = \{p_1 ,\ldots, p_s\}, \quad C = \frac{1}{\ell_{\mathrm{min}}}.
  \]
} \fi For any $A > \frac{1}{\ell_{\mathrm{min}}}$, \ifdefined\SHOWOLDSTABILITY {

  TO BE OMITTED

  Used to be: $\epsilon\in \mathbb Q_{>0}$ } \fi
the $\Omega = (S,A,\vartheta)$-stability
condition  reads:
\begin{itemize}
\item[$\Omega$-\textbf{2}.] for any point $x\in \mathscr C$ that is not a node or marking,
  \[
    \min\{ \frac{{\mathrm{ord}}_x (\rho_i)}{\ell_i} + \frac{1}{\ell_i} \mid i =
    1 ,\ldots, N\} \leq A,
  \]
  \ifdefined\SHOWOLDSTABILITY {

    TO BE OMITTED

    Used to be:
    \[
      \min\{ \frac{{\mathrm{ord}}_x (\rho_i)}{\ell_i} + \frac{1}{\ell_i} -
      \frac{1}{\ell_{\mathrm{min}}} \mid i = 1 ,\ldots, N\} \leq
      \frac{1}{\epsilon},
    \]
  } \fi
\item[$\Omega$-\textbf{3}.] the $\mathbb Q$-line bundle $\mathscr L^{\otimes (-1)} \otimes
  (\omega_{\mathscr C}^{\log})^{\otimes A}$
  \ifdefined\SHOWOLDSTABILITY {

    TO BE OMITTED

    Used to be: $(\mathscr L \otimes \omega_{\mathscr C}^{\log, \otimes
      (1/\ell_{\mathrm{min}})})^{\otimes \epsilon} \otimes
    \omega^{\log}_{\mathscr C}$} \fi is positive,
\end{itemize}
in addition to the prestability condition above, which already covers
$\Omega$-{1} of Definition~\ref{stability}.

Theorem~\ref{thm:intro-separated-Deligne--Mumford-finite-type} implies that
the moduli of $\Omega$-stable LG-quasimaps to $X$ of a fixed degree is a separated
DM stack of finite type.
Replacing $V$ by $\mathbb C^{s}$, viewed as the subspace of $\mathbb C^{N+s}$
defined by $x_1 = \cdots = x_N= 0$,
Theorem~\ref{thm:intro-properness} implies that the substack where the
$\varphi_1 ,\ldots, \varphi_N$ are identically zero is proper.
Note that ${\mathrm{ord}}_x
(\rho_i)$ is always an integer. Hence, when $\ell_{\max} < 2\ell_{\min}$ and $A
\to (\frac{1}{\ell_{\mathrm{min}}})^+$,
$\Omega$-2 is equivalent to that $\rho_1
,\ldots, \rho_s$ does not have a common zero at any $x$.
In particular, the stability in \cite{clader2017landau} is the special case
given by $\ell_1 = \cdots = \ell_s$ and $A \to (\frac{1}{\ell_{\mathrm{min}}})^+$.

\subsection{Comparison with \cite{fan2017mathematical}}
\label{sec:comparison-to-FJR}
We first recall the setup and the main results regarding the geometry of the
moduli spaces in \cite{fan2017mathematical}.

We denote the $V$ in \cite{fan2017mathematical} by $V_{\mathrm{FJR}}$, which is
always assumed to be a vector space.
Indeed, our $V$ will be taken to be
$V_{\mathrm{FJR}}$ or more generally any $Z \subset V_{\mathrm{FJR}}$, to be introduced below.

In \cite{fan2017mathematical}, the input data for the moduli space consists of a
finite dimensional vector space $V_{\mathrm{FJR}}$, two closed subgroups $G, \mathbb
 C^*_{R} \leq GL(V_{\mathrm{FJR}})$, a character $\theta\in \widehat{G}$, such that
\begin{itemize}
\item $(G,V_{\mathrm{FJR}}, \theta)$ satisfies the same assumptions as the
  $(G,V,\theta)$ in this paper;
\item
  $\mathbb C^*_{R}$ is a copy of $\mathbb C^*$, and
  the $\mathbb C^*_{R}$-action on $V_{\mathrm{FJR}}$ commutes with the
  $G$-action, and $G\cap \mathbb C^*_{R}$ is finite.
\end{itemize}
Then for any $G\times \mathbb C^*_R$-invariant subvariety $Z\subset G$ such that
$Z^{\mathrm{s}}(\theta) \neq \emptyset$,
\cite{fan2017mathematical} considered the moduli of prestable LG-quasimaps to
$[Z^{\mathrm{s}}(\theta) /G]$.

Now we introduce the $R$-charged package using our notation:
\begin{itemize}
\item
  $V = Z$;
\item
  $\Gamma\subset GL(V_{\mathrm{FJR}})$ is the subgroup generated by $G$ and
  $\mathbb C^*_{R}$, which naturally act on $V_{\mathrm{FJR}}$;
\item
  $\epsilon: \Gamma \to \mathbb C^*$ is the unique group homomorphism such that
  $\ker(\epsilon) = G$ and $\epsilon|_{\mathbb C^*_R}$ is of the form $z \mapsto
  z^{k}$, $z\in \mathbb C_R^*$, for some $k \in \mathbb Z_{>0}$.
\end{itemize}
For any $\vartheta \in \widehat{\Gamma}$ that is a lift of $\theta$, i.e.,
$\vartheta|_{G} = \theta$, the notion of
prestable LG-quasimaps in this paper coincides with
\cite[Definition~4.2.2]{fan2017mathematical}, which actually does not depend on
the choice of $\vartheta$.

We now compare the stability conditions.
In short, the stability condition of \cite{fan2017mathematical} are
 two special cases of the $\Omega$-stability condition of this paper,
\begin{itemize}
\item
  either when $S$ is full (c.f.~\eqref{eq:being-full}) and every $f \in S$ has slope zero.
\item
  or $A$ tends to $\infty$.
\end{itemize}
The first case is when there exists a ``good lift''. This means the
$\vartheta$-semistable locus coincides with the $\theta$-semistable locus
\cite[p.~237]{fan2017mathematical}.
Good lifts do not always exist. For example, they do not exist for
the Mixed-Spin-P fields
discussed in Section~\ref{sec:brief} or
the LG-phase
of complete intersections of unequal degrees discussed in
Section~\ref{sec:complete-intersection-example}.

Suppose $\vartheta$ is a good lift, we recover the stability
\cite{fan2017mathematical} as follows. Take $S \subset \bigcup_{k\geq 1} R_{k, 0}$  to be
any finite subset of nonzero homogeneous elements that generates
the same ideal as $\bigcup_{k\geq 1} R_{k, 0}$.
Then $S$ is full since
 $\vartheta$ is a good lift.
 For $\epsilon \in \mathbb Q_{>0}$, set $A = \frac{1}{\epsilon}$.
 Then the $(\epsilon, \vartheta)$-stability defined in
 \cite[Definitions~4.2.11,~4.2.13]{fan2017mathematical} is equivalent to the
 $\Omega = (S,A,\vartheta)$-stability in Definition~\ref{stability}.

When there is no good lift, \cite{fan2017mathematical}
 also defined the $(\epsilon = 0^+)$-stability
condition which disallows any rational tails but has no restriction on the
length of base points. It coincides with the $\Omega$-stability
condition for $A\gg 0$ (c.f.\ Corollary~\ref{cor:stability-condition-A-infinity}).
This is in contrast to the theory of stable MSP fields, where no base points are allowed.
The nonvanishing of $(\mu, \nu)$ for stable MSP fields plays an essential role
in separating the ``Feynman sum part'' and ``holomorphic ambiguity part'' in
the application of MSP theory to higher-genus Gromov-Witten invariants.

\subsection{Stable quasimaps}
\label{sec:stable-quasimaps-as-a-special-case}
Given $(G,V,\theta)$ as in the quasimap setting, take $\Gamma = G \times \mathbb
C^*$ and let $\Gamma$ act on $V$ via its projection onto $G$. Take $\vartheta$
to be the pullback of $\theta$. Thus $\mathbb C^*$ acts on $V$ trivially and all
the nonzero functions have $R$-charge $0$. Take any $S$ to be any finite set of
nonzero homogeneous generators for the ideal of the $\theta$-unstable locus and
take $A = \frac{1}{\epsilon}$, $\epsilon \in \mathbb Q_{>0}$. It is easy to see
that we recover the theory of $\epsilon$-stable quasimaps to $X$ in
\cite{ciocan2014stable, cheong2015orbifold}.

\section{Proof of main theorems: the projective space case}
\label{sec:projective-space-case}
By a standard argument, the moduli of prestable LG-quasimaps is an Artin stack,
locally of finite presentation over $\mathbb C$. Indeed, by
\cite[Lemma~2.8]{cheong2015orbifold} (c.f.\
\cite{lieblich2006remarks,abramovich2011twisted,ollson2006hom}),
the category of representable principal
$\Gamma$-bundles on twisted curves is an Artin stack locally of finite presentation over
$\mathbb C$. Then one may embed $V$ into a linear representation and use the
cone construction in  \cite{chang2012gromov}.

In this section we prove
Theorems~\ref{thm:intro-separated-Deligne--Mumford-finite-type}
and~\ref{thm:intro-properness} when $X$ is a projective space with the $R$-charged
package described in Section~\ref{sec:projective-space-example}. We will reduce
the general case to this special case in the next section.

We follow the setup in Section~\ref{sec:projective-space-example}.

\subsection{It is a Deligne-Mumford stack}
\begin{lemm}
  \label{lem:finite-automorphisms-P-N-1-case}
  For any $\xi\in \mathrm{LGQ}^{\Omega}_{g,k}(\mathbb P^{N-1},d)(\mathbb C)$, $\mathrm{Aut}(\xi)$ is finite.
\end{lemm}
\begin{proof}
  First observe that $\mathrm{Aut}(\xi)$ acts faithfully on $(\mathscr C,
  \Sigma)$, since on each irreducible component of $\mathscr C$ at least one $\varphi_i$ is nontrivial.
  Dropping those
  $\varphi_i$ that are identically zero, without loss of generality we assume
  none of the $\varphi_i$'s are identically zero.

  Let us consider the case $(\mathscr C, \Sigma^{\mathscr C}) \cong (\mathbb
  P^1, \emptyset)$. Set $\textstyle D = \bigcup_{i=1}^N (\varphi_i =
  0)^{\mathrm{red}} \subset \mathscr C$.  If $\# D > 2$, then it is clear that
  $\mathrm{Aut}(\xi)$ is finite, since $D$ is left invariant by
  $\mathrm{Aut}(\xi)$ and $\mathrm{Aut}(\xi)$ acts faithfully on $\mathscr C$.
  Thus we assume $\#D \leq 2$ from now on.

  We first show that $\#D = 2$.
  Indeed, if not, say $D\subset \{x\}$, then for each $i$,
  \[
    \mathrm{ord}_{x} \varphi_i
    = \deg(\mathscr L \otimes (\omega_{\mathscr C}^{\log})^{\otimes c_i})
    = \deg(\mathscr L) - 2 c_i.
  \]
  By $\Omega$-2, for some $i_0$ we have $\mathrm{ord}_{x} \varphi_{i_0}  +
  c_{i_0} \leq A$.
  Hence $\deg(\mathscr L) \leq A + c_{i_0}$.
  But $\Omega$-3 says $\deg(\mathscr L \otimes (\omega^{\log}_{\mathscr
    C})^{\otimes A}) > 0$,
  which is $\deg(\mathscr L) > 2 A$.
  Thus $c_{i_0} > A$, contradicting to our assumption that $A > \max_{i} \{c_i\}$.
  This proves the claim.

  We next show that $\#D = 2$ also violates the stability.
  Let $z$ be the standard coordinate on $\mathscr C \cong \mathbb P^1$,
  and we may assume $D = \{0,\infty\}$.
  Fix $m\in \mathbb Z_{>0}$ such that $mA \in \mathbb Z$ and set
  \[
    \phi_i = \varphi^m_i \cdot (\frac{\operatorname{d}\! z}{z})^{m(A - c_i)},
    \quad i = 1 ,\ldots, N.
  \]
  They are all rational sections of the same line bundle
  \[
    \mathscr N : = \mathscr L^{\otimes m}
    \otimes (\omega_{\mathscr C}^{\log})^{\otimes mA}.
  \]

  Since $\mathrm{Aut}(\mathscr C, (0,\infty)) \cong \mathbb C^*$ leaves
  $\frac{\operatorname{d}\! z}{z}$ invariant, we see that the identity component
  $\mathrm{Aut}^\circ(\xi)$ of $\mathrm{Aut}(\xi)$ is a subgroup of
  \begin{equation}
    \label{eq:aut-phi}
    \mathrm{Aut}(\mathscr C, (0,\infty), \mathscr N, \phi_1 ,\ldots, \phi_N).
  \end{equation}
  Hence we only need to show \eqref{eq:aut-phi} is a finite group. For this,
  it suffices to show that all the $\phi_i$'s are
  not all proportional to each other.
  Indeed,
  by $\Omega$-3,
  \[
    \deg(\mathscr N) = m \deg(\mathscr L \otimes (\omega_{\mathscr
      C}^{\log})^{\otimes A}) > 0.
  \]
  Thus if all the $\phi_i$'s are proportional to each other, they must have a common zero.
  However, by $\Omega$-2, for $x = 0$ we have
  \[
    \min_i \{\mathrm{ord}_x \phi_i \} = m \min_{i}\{\mathrm{ord}_{x}\varphi_i + c_i\}
    - mA \leq 0.
  \]
  This means that at $x = 0$, at least one $\phi_i$ is nonzero or has a pole.
  And the same holds true for $x = \infty$.
  Thus the proof of the lemma is complete in the case
  $(\mathscr C, \Sigma^{\mathscr C}) \cong (\mathbb
  P^1, \emptyset)$. The other cases are similar and we leave them to the reader.
\end{proof}
\begin{lemm}
  \label{lem:DM-type-projective-space-case}
  The stack $\mathrm{LGQ}^{\Omega}_{g,k}(\mathbb P^{N-1},d)$ is a DM stack of finite type.
\end{lemm}
\begin{proof} Being a DM stack follows
  from the finiteness of automorphism
  groups (c.f.~\cite[Remark~8.3.4]{olsson2016algebraic}); the finite type part is
  standard, using $\Omega$-{3}.
\end{proof}
\subsection{Introducing curves with nonreduced components}
\label{sec:nonreduced-fiber}
We will use the valuative criterion to prove the properness of the
moduli.
To this end, we introduce more notation.
Let $T$ be the spectrum of a
discrete valuation ring with residue field $\CC$ and fractional field $K$.
We denote by $\eta$ and $\zeta$ the generic and closed points of $T$,
respectively, and let $t$ be a uniformizer of $T$.

As an intermediate step of the proof, we
introduce flat families of pointed curves
$\pi: \Sigma^{\mathscr C} \subset \mathscr C \to T$ with possibly nonreduced special fiber.
More precisely, when $\mathscr C_{\zeta}$ is nonreduced, we require
\begin{itemize}
\item[(C1)] $\mathscr C$ is a regular surface. In particular $\mathscr C_{\eta} \to \eta$ is smooth.
\item[(C2)]
  Let $q\in \sC_{\zeta}$ be a singular point of $(\mathscr
  C_{\zeta})^{\mathrm{red}}$ and $\hat \sC$ be the formal completion
  along $q$, then $\hat\sC\cong \operatorname{Spec} \mathbb C
  \llbracket  x,y,t  \rrbracket /(x^ay^b-t)$, $a,b >0$.
\end{itemize}

We call such a  $q$ a node of $\sC$.
Let $\Si^{\sC}\subset \sC$ be the markings, by definition it is a disjoint union of
sections of $\pi$, and thus is disjoint from the non-reduced components of $\sC_{\zeta}$.

In this situation, we will replace $\omega_{\mathscr C}$ by
$\Omega_{\sC/T}(\log \sC_\zeta)$, the sheaf of relative log-differentials of
$(\sC,\sC_\zeta)$ relative to $(T, \zeta)$.
Namely, we redefine
\[
  \omega_{\sC/T}^{\log}:=
  \Omega_{\sC/T}(\log \sC_\zeta) \otimes \mathcal O_{\mathscr C} (\Sigma^{\mathscr C})
  = \Omega_{\sC/T}(\log \sC_\zeta+\Sigma^{\mathscr C}).
\]
Thus, we will have
\[
  \varphi_i \in H^0(\mathscr C, \mathscr L \otimes
  (\omega^{\log}_{\sC/T})^{\otimes c_i})
  =
  H^0(\mathscr C, \mathscr L \otimes (\Omega_{\sC/T}(\log \sC_\zeta +
  \Sigma^{\mathscr C}))^{\otimes c_i}).
  \footnote{Recall that $\Omega_{\sC/T}(\log \sC_\zeta)$
    is the sheaf of meromorphic differentials on $\mathscr C$ with at worst log-poles along the
    irreducible components of $\sC_{\zeta}$, modulo $\pi^* \operatorname{d}\! \log t$.
    In the local model $\hat\sC\cong \operatorname{Spec} \mathbb C \llbracket  x,y,t\rrbracket/(x^ay^b-t)$
    with $a$, $b>0$,
    $\omega^{\log}_{\sC/T}|_{\hat\sC}$ is generated by $a \operatorname{d}\!
    \log x = -b \operatorname{d}\! \log y$.
    In case $\mathscr C_{\zeta}$ is reduced, $\omega^{\log}_{\sC/T}$ is the usual
    $\omega^{\log}_{\sC/T}$. Hence the new definition is compatible with the old one.}
\]

We now study the change of $\Omega_{\sC/T}(\log \sC_\zeta)$ under blowups. Let
$\pi: \Sigma^{\mathscr C} \subset \mathscr C \to T$ satisfy (C1) and (C2) as above,
and let $f: \mathscr C^\prime \to \mathscr C$ be the blowup at a closed point
$q\in \mathscr C_{\zeta}$. Let $\mathscr E$ be the exceptional divisor and
$\Sigma^{\mathscr C^{\prime}}$ be the proper transform of $\Sigma^{\mathscr C}$.
Consider the natural isomorphism
\[
  df_\eta: f^*\omega^{\log}_{\sC_\eta/\eta}\lra
  \omega^{\log}_{\sC'_\eta/\eta}.
\]

\begin{lemm} \label{lem:Omega-log-under-blowup}
  If $q$ is a marking or node, then $df_{\eta}$ extends to an isomorphism
  $df: f^*\omega^{\log}_{\sC/T}\to \omega^{\log}_{\sC'/T}$.
  Otherwise,  it extends to an isomorphism
  $df: f^*\omega^{\log}_{\sC/T} \to \omega^{\log}_{\sC'/T}(- \mathscr E)$.
\end{lemm}

\begin{proof}
  This follows from a local computation, and will be omitted.
\end{proof}

At some point, we will make the special fiber reduced by taking a finite base
change and normalization. We now study the change of $\Omega_{\sC/T}(\log
\sC_\zeta)$ in such procedure.
Let $T^\prime \to T$ be a finite cover totally ramified at $\zeta$, and let
$\eta^\prime$ (resp.~$\zeta^\prime$) be the generic (resp. closed) point of
$T^\prime$.

\begin{lemm}
  \label{lem:Omega-log-under-base-change-and-normalization}
  Suppose we have a commutative diagram
  \[
    \begin{tikzcd}
      \Sigma^{\mathscr C^\prime}\sub \mathscr C^\prime \ar[r,"f"]
      \ar[d,"\pi^\prime"']& \Sigma^\sC\sub \mathscr C \ar[d, "\pi"]\\
      T^\prime \ar[r] & T,
    \end{tikzcd},
  \]
  where $\pi^\prime$ is a family of pointed nodal (reduced) curves.
  Suppose $f$ is finite, $f|_{\mathscr C^\prime_{\eta^\prime}}$ is \'etale,
  and $f^{-1}(\Sigma^{\mathscr C}) = \Sigma^{\mathscr C^\prime}$,
  then the pullback of differentials extends to an isomorphism
  \[
    df: f^*\omega^{\log}_{\sC/T}\lra \omega_{\sC'/T^\prime}^{\log}.
  \]
\end{lemm}

\begin{proof}
  This follows from the local computation.
\end{proof}
\subsection{Extending the fields to a given family of curves}
We continue to use the notation $T, \eta, \zeta$ as before.
Consider a family of pointed curves $\Sigma^{\mathscr C} \subset \mathscr C \to T$
such that $\Sigma^{\mathscr C}$ is contained in the relative smooth locus, and that
\begin{enumerate}
\item[Case] 1: when  $\mathscr C_{\zeta}$ is reduced,
  we require that
  $\mathscr C_{\eta}$ is smooth and $\mathscr C_{\zeta}$ is nodal.
\item[Case] 2: when $\mathscr C_{\zeta}$ is non-reduced,  we require (C1) and (C2) in the previous subsection,
\end{enumerate}

In both cases, we have defined $\omega^{\log}_{\mathscr C/T}$.
In either case, given a pair $(\mathscr L_{\eta}, \varphi_{\eta})$,
where $\mathscr L_{\eta}$ is an invertible sheaf of $\sO_{\sC_\eta}$-modules,
and $0 \neq \varphi_{\eta} \in H^0(\mathscr C_{\eta},
\bigoplus_{i=1}^{N}\mathscr L_{\eta} \otimes (\omega^{\log}_{\mathscr
C_{\eta}/\eta})^{\otimes c_i})$,
as the first step we would like to extend this pair to the pair $(\mathscr L,
\varphi)$,
where $\mathscr L$ is a line invertible sheaf of $\sO_{\sC}$-modules,
and $\varphi \in H^0(\mathscr C, \bigoplus_{i=1}^{N}\mathscr L \otimes
(\omega^{\log}_{\mathscr C/T})^{\otimes c_i})$, such that $(\varphi = 0)$ contains no irreducible component
of $\mathscr C_{\zeta}$.

\begin{lemm} \label{lem:extending-L-and-varphi}
  In either case, if the extension $(\mathscr L, \varphi)$ exists,
  it is unique up a unique isomorphism.
  In Case 2, the extension always exists.
\end{lemm}
\begin{proof}
  In either case, $\mathscr C$ is normal. Hence any two extensions of
  $\mathscr L_{\eta}$ to line bundles on $\mathscr C$ differ by some $\mathcal
  O_{\mathscr C}(\mathscr D)$, where $\mathscr D$ is a Cartier divisor with
  $\mathrm{supp}(\mathscr D) \subset \mathscr C_{\zeta}$. In Case 2, $\mathscr
  C$ is regular, and thus every line bundle on $\mathscr C_{\eta}$ extends.
  The rest is obvious.
\end{proof}

\subsection{Prestable extensions with base point condition}
We use the valuative criterion.
For any
\[
  \xi_{\eta} = (\sC_\eta,\Si^{\sC_{\eta}},\sL_{\eta},\varphi_{\eta}) \in
  \mathrm{LGQ}^{\Omega}_{g,k}(\mathbb P^{N-1}, d)(\eta),
\]
we intend to show that possibly after a finite base change, $\xi_\eta$
extends uniquely to a $\xi = (\sC,\Si^\sC,\sL, \varphi)\in
\mathrm{LGQ}^{\Omega}_{g,k}(\mathbb P^{N-1}, d)(T)$.

By a standard gluing argument, we may assume that $\sC_{\eta} \to \eta$ is smooth with
connected geometric fibers. Indeed, along nodes or
markings, $\omega_{\mathscr C/T}^{\log}$ is canonically
trivial and $\varphi$ is nonzero. Thus
we can glue the LG-quasimaps in the same way
as we glue maps from curves into $\mathbb P^{N-1}$.

\begin{lemm} \label{lem:prestable-reduction}
Let $\xi_\eta=(\sC_\eta,\Si^{\sC_\eta},\sL_\eta, \varphi_\eta)$ be as above,
with $\sC_\eta$ regular.
Then possibly after a finite base change, there is an extension
  \[
    \textstyle  \xi^\prime =
    (\sC^\prime,\Si^{\sC^\prime},\sL^\prime, \varphi^\prime)\in
    \mathrm{LGQ}^{\mathrm{pre}}_{g,k}(\mathbb P^{N-1}, d)(T)
  \]
satisfying the base-point condition $\Omega$-2. Moreover, given any extension of
nodal curves $\sC^{\prime\prime} \to T$ of $\sC_\eta$, we can choose $\mathscr
C^{\prime}$ so that there exists a $T$-morphism $\mathscr C^\prime \to \mathscr
C^{\prime\prime}$ extending the identity on $\mathscr C_{\eta}$.
\end{lemm}

\begin{proof}
  The construction is divided into two steps. In the first step, we relax the
  requirement by allowing nonreduced special fibers. We find an extension
  satisfy (G1)-(G5), to be introduced below.
  In the second step, we make the special fiber reduced by a
  standard process of finite base change and normalization, and the resulting
  family will be the desired $\xi^\prime$.

  Now possibly after finite base change, let
  \[
    \xi=(\sC,\Si^{\sC},\sL, \varphi)
  \]
  be any extension of $\xi_\eta$ with possibly nonreduced special fiber, as
  discussed in Section~\ref{sec:nonreduced-fiber}. 

  \begin{subl}
    Suppose $\xi$ satisfies
    \begin{enumerate}
    \item[(G1)] The $T$-flat part $\mathscr D$ of $\bigcup_{i = 1}^N(\varphi_i =
      0)^{\mathrm{red}}$ is a disjoint union of sections.
    \item[(G2)]
      For any marking $\Sigma_{i}^{\mathscr C}$, either
      $\Sigma_{i}^{\mathscr C}\subset \mathscr D$ or $\Sigma_{i}^{\mathscr C}\cap
      \mathscr D = \emptyset$.
    \item[(G3)]
      For any two intersecting irreducible components $\mathscr C_{\zeta, 1},
      \mathscr C_{\zeta,2} \subset (\mathscr C_{\zeta})^{\mathrm{red}}$, there is a
      $j$ so that $(\varphi_j = 0)$ does not contain $\mathscr
      C_{\zeta, 1}\cup \mathscr C_{\zeta,2}$.
    \item[(G4)]
      For any irreducible component $\mathscr C_{\zeta, 1} \subset (\mathscr
      C_{\zeta})^{\mathrm{red}}$ intersecting a marking $\Sigma^{\mathscr C}_i$, there exists
      some $j$ such that $(\varphi_j = 0)$ does not contain $\mathscr C_{\zeta, 1}
      \cup \Sigma^{\mathscr C}_i$.
    \item[(G5)]
      For any irreducible component $\mathscr D_1 \subset \mathscr D$
      intersecting any irreducible component $\mathscr C_{\zeta,1}\subset (\mathscr C_{\zeta})^{\mathrm{red}}$, suppose
      $\mathscr D_1\not\subset \Sigma^{\mathscr C}$, then there exists some $j$ such
      that $(\varphi_j = 0)$ does not contain $\mathscr C_{\zeta,1}$, and that
      \[ \mathrm{ord}_{\mathscr D_1} \varphi_j =
        \min_{\alpha} \{
        \mathrm{ord}_{\mathscr D_1} \varphi_\alpha  + c_\alpha
        \}-c_j,
      \]
      where $\mathrm{ord}_{\mathscr D_1} \varphi_\alpha$ means the vanishing
      order of $\varphi_j$ along $\mathscr D_1$.
    \end{enumerate}
    Then by finite base change and normalization
    we can construct a $\xi^\prime$ from $\xi$ to satisfy the requirement of Lemma~\ref{lem:prestable-reduction}.
  \end{subl}
  \begin{proof}
    We first claim that $(\varphi = 0) \cap \mathscr C_{\zeta}$ is always discrete.
    Indeed, by our construction,
    it is automatic that $(\varphi = 0)$
    does not contain any irreducible component of $(\mathscr C_{\zeta})^{\mathrm{red}}$.

    Also, (G1) and (G3) imply that $(\varphi = 0)$ is away from the nodes of
    $(\mathscr C_{\zeta})^{\mathrm{red}}$. And (G1), (G2) and (G4) imply that
    $(\varphi = 0) \cap \Sigma^{\mathscr C} = \emptyset$.

    Further, (G1) and (G5) implies that for any $q\in (\mathscr C_{\zeta})^{\mathrm{red}}$ that is not a node or marking,
    we have
    \[
      \min_{i}\{\mathrm{ord}_{q} (\varphi_i|_{(\mathscr C_\zeta)^{\mathrm{red}}}) + c_i\} \leq A.
    \]
    Indeed, if $q\not\in \mathscr D$, then $\varphi(q) \neq 0$ by the
    construction of $\varphi$ and $\mathscr D$, and the above follows from the
    fact that $c_{i} < A$ for each $i$. Otherwise let
    $\mathscr D_1$ and $\mathscr C_{\zeta,1}$ be as in (G5), and suppose $q = \mathscr D_1\cap
    \mathscr C_{\zeta,1}$, then by the $\Omega$-stability of the generic fiber,
    we have
    \[
      \min_{i} \{
      \mathrm{ord}_{\mathscr D_1} \varphi_i  + c_i\} \leq A.
    \]
    Combined with (G1) and (G5), we obtain
    \[
      \mathrm{ord}_{q} (\varphi_j|_{(\mathscr C_{\zeta})^{\mathrm{red}}}) + c_j =
      \mathrm{ord}_{\mathscr D_1} \varphi_j + c_j= \min_{\alpha} \{
      \mathrm{ord}_{\mathscr D_1} \varphi_\alpha  + c_\alpha
      \} \leq A.
    \]
    Summarizing,
   $\xi$ is almost the desired extension
    except that $\mathscr C_{\zeta}$ might be nonreduced. Then we use a
    standard procedure of finite base
    change and normalization to make the special fiber reduced, resulting in $(\mathscr C^\prime,
    \Sigma^{\mathscr C^{\prime}})$.
    We set  $(\mathscr L^\prime, \varphi^\prime)$ to
    be the pullback of $(\mathscr L, \varphi)$ to $\mathscr C^\prime$.
    This is possible
    thanks to
    Lemma~\ref{lem:Omega-log-under-base-change-and-normalization}.
    Then $\xi^\prime := (\mathscr C^\prime,
    \Sigma^{\mathscr C^{\prime}}, \mathscr L^\prime, \varphi^\prime)$
    satisfies the requirement of Lemma~\ref{lem:prestable-reduction}.
  \end{proof}
To finish the proof of the lemma, we need to construct a $\xi$ satisfying
(G1)-(G5).

Without loss of generality, we assume that none of the $\varphi_{\eta, i}$'s are trivial.
Set $\mathscr D_{\eta} = \bigcup_{i} (\varphi_{\eta, i} = 0)^{\mathrm{red}} \subset \mathscr C_{\eta}$.
After finite base change, we can make $\mathscr D_{\eta}$ a union of sections of
$\mathscr C_{\eta} \to \eta$.
By viewing $\mathscr D_{\eta}$ as additional markings, possibly after finite
base change,
we can extend $(\mathscr C_{\eta}, \Sigma^{\mathscr C_{\eta}} \cup \mathscr D_{\eta})$ to
a $T$-family of pointed nodal curves $(\mathscr C, \Sigma^{\mathscr C}\cup
 \mathscr D)$ such that $\mathscr C$ is a regular surface.
Using Lemma~\ref{lem:extending-L-and-varphi},
extend $(\mathscr L_{\eta}, \varphi_{\eta})$ to $(\mathscr L, \varphi)$ on
$\mathscr C$ such that $(\varphi = 0)$ does not contain any irreducible
component of $\mathscr C_{\zeta}$.

Consider the $T$-family $\xi = (\mathscr C, \Sigma^{\mathscr C}, \mathscr L,
\varphi)$.
We now modify $\xi$ by a sequence of blowups.
After each blowup, we obtain a unique extension of
$(\mathscr L_{\eta}, \varphi_{\eta})$ to the new
surface by Lemma~\ref{lem:extending-L-and-varphi}. This way we get a new
extension and we replace the original $\xi$ by it. In the following, we will show that after
finitely many such steps, $\xi$ satisfies (G1)-(G5) and the proof will thus be complete.

By our construction, $\xi$ already satisfies (G1) and (G2).
As further blowing ups will not affect the properties
(G1) and (G2), they will still be satisfied by the subsequent $\xi$'s.

We first achieve property (G5).
Let $\mathscr D_1, \mathscr C_{\zeta,1}$ be as in the statement of (G5).
For $i = 1 ,\ldots, N$, set
\[
  \textstyle \lambda_i = \mathrm{ord}_{\mathscr D_1} \varphi_i  + c_i -
  \min_{\alpha = 1}^N\{ \mathrm{ord}_{\mathscr D_1} \varphi_{\alpha}  + c_\alpha\}
  \quad \text{and} \quad
  \mu_{i} = \mathrm{ord}_{\mathscr C_{\zeta, 1}} \varphi_{i}.
\]
Note that $\lambda_i,\mu_i \in \mathbb Z_{\geq 0}$ and
$\min_i \{\lambda_i\} = \min_i \{\mu_i\} =0$.
We let
\[
  \delta(\mathscr D_1, \mathscr C_{\zeta,1}) := \min_{\lambda_i= 0} \mu_i+
  \min_{\mu_i = 0} \lambda_i.
\]
Note that $\delta(\mathscr D_1, \mathscr C_{\zeta,1}) \in \mathbb Z_{\geq 0}$, and that
$\delta(\mathscr D_1, \mathscr C_{\zeta,1})  = 0$ if and only if
$\lambda_i= \mu_i= 0$ for some $i$, if and only if (G5)
holds for $(\mathscr D_1, \mathscr C_{\zeta,1})$.

From now on we suppose $\delta(\mathscr D_1, \mathscr C_{\zeta,1}) > 0$.
Recall that $\mathscr C$ is a regular surface.
Let $\tau:\tilde{\mathscr C} \to \mathscr C$ be the blowing up at
$\mathscr D_1\cap\mathscr C_{\zeta,1}$, with exceptional divisor $\mathscr E$. Let
$\tilde{\mathscr D}_1$ be the proper transform of $\mathscr D_1$, and
$\Sigma^{\tilde{\mathscr C}}$ be the proper transform of $\Sigma^{\mathscr C}$.
Using Lemma~\ref{lem:Omega-log-under-blowup},
\begin{align*}
  \tau^* \varphi_i \in & H^0(\tilde{\mathscr C}, \tau^*{\mathscr L} \otimes  (\tau^*\omega^{\log}_{{\mathscr C}/T}
                         (-(\mathrm{ord}_{\mathscr D_1}\varphi_i +
                         \mathrm{ord}_{\mathscr C_{\zeta, 1}} \varphi_i)\mathcal E)) \\
  = & H^0(\tilde{\mathscr C}, \tau^*{\mathscr L} \otimes
      (\tau^*\omega^{\log}_{{\mathscr C}/T}
      (-(\mathrm{ord}_{\mathscr D_1}\varphi_i + \mathrm{ord}_{\mathscr
      C_{\zeta, 1}} \varphi_i + c_i)\mathcal E)
      ).
\end{align*}
Set
$\tilde {\mathscr L} =
\tau^*\mathscr L (- \min_{\alpha}\{\mathrm{ord}_{\mathscr D_1}\varphi_\alpha
+ \mathrm{ord}_{\mathscr C_{\zeta, 1}} \varphi_\alpha + c_\alpha \} \mathscr E)$
and let
\[
  \tilde{\varphi}_i \in
  H^0(\tilde{\mathscr C}, \tilde{\mathscr L} \otimes
  (\Omega_{\tilde{\mathscr C}/T}(\log \tilde{\mathscr C}_{\zeta}
  + \Sigma^{\tilde{\mathscr C}}
  ))^{\otimes c_i}
  )
\]
be the image of $\tau^*\varphi_i$ under the natural inclusion of sheaves. Then we have
\begin{align*}
  \mathrm{ord}_{\mathcal E} \tilde{\varphi}_i
  = & \mathrm{ord}_{\mathscr D_1}\varphi_i + \mathrm{ord}_{\mathscr
    C_{\zeta, 1}} \varphi_i + c_i
    -
    \min_{\alpha}\{
    \mathrm{ord}_{\mathscr D_1}\varphi_\alpha + \mathrm{ord}_{\mathscr
    C_{\zeta, 1}} \varphi_\alpha + c_\alpha
    \} \\
  = & \lambda_i + \mu_i - \min_{\alpha}\{ \lambda_{\alpha} + \mu_{\alpha} \},
\end{align*}
and
$\mathrm{ord}_{\tilde{\mathscr D}_1} \tilde{\varphi}_i =
\mathrm{ord}_{\mathscr D_1} \varphi_i = \mu_{i}$.
This means $(\tilde{\varphi} = 0)$ does not contain $\mathscr E$ and
hence $(\tilde{\mathscr L}, \tilde{\varphi})$ is precisely
the extension in Lemma~\ref{lem:extending-L-and-varphi}.

We will replace $(\mathscr C, \Sigma^{\mathscr C}, \mathscr L, \varphi)$ by
$(\tilde{\mathscr C}, \Sigma^{\tilde{\mathscr C}}, \tilde{\mathscr L},
\tilde{\varphi})$.
To show that (G5) can be achieved in finitely many such steps, it suffices to show
that
\[
  \delta(\tilde{\mathscr D}_1, \mathcal E) <
  \delta(\mathscr D_1, \mathscr C_{\zeta,1}),
\]
where
\begin{align*}
  \delta(\tilde{\mathscr D}_1, \mathcal E) =
  & \min \{\mu_{\beta} \mid \lambda_{\beta} + \mu_{\beta} =
    \min_{\alpha}\{\lambda_{\alpha} + \mu_{\alpha}\}\} \\
  & + \min \{
    \lambda_{\beta} + \mu_{\beta} - \min_{\alpha}\{\lambda_{\alpha} + \mu_{\alpha}\}
    \mid \mu_{\beta} = 0
    \}.
\end{align*}
It is equivalent to
\begin{equation}
  \label{eq:effect-of-blowup-concrete-form}
  \min_{ \lambda_{\beta} + \mu_{\beta}
    = \min_{\alpha}\{ \lambda_{\alpha} + \mu_{\alpha}\}}
  \mu_{\beta}
  - \min_{\alpha}\{\lambda_{\alpha} + \mu_{\alpha}\}
  < \min_{ \lambda_{\alpha} = 0} \mu_{\alpha} .
\end{equation}
Note that we have
  \[
    \min_{\lambda_{\beta} + \mu_{\beta}
    = \min_{\alpha}\{\lambda_{\alpha} + \mu_{\alpha}   \}}
    \mu_{\beta}
  \leq  \min_{\alpha}\{\lambda_{\alpha} + \mu_{\alpha} \}
    \leq \min_{\lambda_{\alpha} = 0} \mu_{\alpha} .
  \]
Further, since $\delta(\mathscr D_1, \mathscr C_{\zeta,1}) > 0$, $\lambda_{\alpha} =
  \mu_{\alpha} = 0$ for no $\alpha$. Hence $\min_{\alpha}\{\lambda_{\alpha} +
  \mu_{\alpha}\} > 0$.
These two combined give \eqref{eq:effect-of-blowup-concrete-form}.
Hence (G5) can be achieved after finitely many blowing ups.

The goals (G3) and (G4) can be dealt with using the same method,
leaving out the $c_i$'s from the formulas: (G3) (resp.\ (G4)) can be achieved by successively
blowups at $\mathscr C_{\zeta,1} \cap \mathscr C_{\zeta,2}$ (resp.\ $\mathscr
C_{\zeta, 1} \cap \Sigma^{\mathscr C}_i$). Note that
by (G1) and (G2), the three types of modifications do not affect each other and
can be done in any order. Thus the proof is complete.
\end{proof}

\subsection{Stabilization}
Let $T,\eta,\zeta$, and $\xi_{\eta} = (\mathscr C_{\eta}, \Sigma^{\mathscr
  C_{\eta}}, \mathscr L_{\eta}, \varphi_{\eta})$ be as before.
We will show that any prestable extension of $\xi_{\eta}$
has a unique stabilization. This will finish the existence part of the valuative
criterion. Since we have shown that any two prestable extensions are dominated
by a third one, this also implies the uniqueness part.

\medskip

Now let
\[
  \xi^{\prime},\,  \xi^{\prime\prime} \in \mathrm{LGQ}^{\mathrm{pre}}_{g,k}(\mathbb P^{N-1}, d)(T)
\]
be two extensions of $\xi_\eta$.
We write $\xi^{\prime} = (\sC^{\prime},\Si^{\sC^{\prime}},\sL^{\prime},
\varphi^{\prime})$ and similarly for $\xi^{\prime\prime}$.
Due to Lemma~\ref{lem:extending-L-and-varphi}, $\xi_{\eta}$ is completely
determined by $ (\sC^{\prime},\Si^{\sC^{\prime}})$.
\begin{defi}
  We say $\xi^{\prime\prime}$ is a contraction of $\xi^\prime$, and
  write $f:\xi^{\prime} \to \xi^{\prime\prime}$, if there exists a $T$-morphism
  $f: \mathscr C^\prime \to \mathscr C^{\prime\prime}$ extending the identity on $\mathscr
  C_{\eta}$. We call it an ($\Omega$-)stabilization if in addition
  $\xi^{\prime\prime}$ is ($\Omega$-)stable.
\end{defi}
The goal of
this subsection is to prove
\begin{lemm}
  \label{lem:stabilization}
  If $\xi^\prime$ satisfies $\Omega$-2, then it has a unique stabilization.
\end{lemm}
The proof is similar to the quasimap case, with one additional complication: The
degree of a quasimap on each irreducible component is always non-negative, but
$\mathscr L^\prime \otimes (\omega_{\mathscr C^\prime/T}^{\log})^{\otimes A}$
can be negative on some irreducible components of $\mathscr C^\prime_{\zeta}$.

In a pointed nodal curve, a rational tail (resp. bridge) is defined to be a
smooth rational irreducible component with one (resp. two) special points.

\begin{lemm}
  \label{lem:contracting-a-rational-tail}
  Let $\mathscr E\subset \mathscr C^{\prime}_{\zeta}$
  be a rational tail. Then there exists a contraction
  $f:\xi^\prime\to \xi^{\prime\prime}$ contracting $\mathscr E$ to some point $q$.
  Further, $\xi^{\prime\prime}$ violates $\Omega$-2 at $q$ if and
  only if $\xi^{\prime}$ satisfies $\Omega$-3 on $\mathscr E$.
\end{lemm}

\begin{proof}
  Let $f: \mathscr C^\prime \to \mathscr C^{\prime\prime}$ be the contraction
  of $\mathscr E$ to some point $q$. Since $q$ is a regular point of $\mathscr
  C^{\prime\prime}$, $(\mathscr L^\prime, \varphi^\prime)|_{f\upmo(\mathscr
    C^{\prime\prime} \setminus \{q\})}$ uniquely extends to $\mathscr
  C^{\prime\prime}$. This proves the existence part.

  For the ``further'' part, let $\mathscr C^{\prime}_{\zeta,1} =
  \overline{\mathscr C^\prime_{\zeta} \setminus \mathscr E}$ and $\tilde{q} = \mathscr
  C^\prime_{\zeta,1}\cap \mathscr E$. Then $f$ restricts to an isomorphism
  $f_1:= f|_{\mathscr C^{\prime}_{\zeta,1}} : \mathscr C^{\prime}_{\zeta,1} \to \mathscr
  C^{\prime\prime}_{\zeta}$.
  By construction we have
  $(\mathscr L^{\prime}, \varphi^\prime)|_{\mathscr C^{\prime}_{\zeta,1} \setminus
    \{\tilde{q}\}} \cong (f_1^*\mathscr L^{\prime\prime},
  f_1^*\varphi^{\prime\prime})|_{\mathscr C^{\prime\prime}_{\zeta} \setminus \{q\}}$.
  Hence for each $i$ we have
  \begin{align*}
    & \mathrm{ord}_{\tilde{q}}(\varphi_i^{\prime}|_{\mathscr
      C^{\prime}_{\zeta,1}}) -
      \mathrm{ord}_{q}(\varphi_i^{\prime\prime}|_{\mathscr C^{\prime\prime}_{\zeta}}) \\
    = & \deg(\mathscr L^{\prime} \otimes (\omega_{\mathscr
        C^\prime/T}^{\log})^{\otimes c_i}|_{\mathscr C^{\prime}_{\zeta,1}})
        -
        \deg(\mathscr L^{\prime\prime} \otimes
        (\omega_{\mathscr C^{\prime\prime}/T}^{\log})^{\otimes c_i}|_{\mathscr C^{\prime\prime}_{\zeta}}) \\
    = & -\deg(\mathscr L^{\prime} \otimes (\omega_{\mathscr
        C^\prime/T}^{\log})^{\otimes c_i}|_{\mathscr E}),
  \end{align*}
  That is,
  \[
    \mathrm{ord}_{q}(\varphi_i^{\prime\prime}|_{\mathscr
      C^{\prime\prime}_{\zeta}}) + c_i = A +
    \mathrm{ord}_{\tilde{q}}(\varphi_i^{\prime}|_{\mathscr
      C^{\prime}_{\zeta,1}}) + \deg(\mathscr L^{\prime} \otimes
    (\omega_{\mathscr C^\prime/T}^{\log})^{\otimes A}|_{\mathscr E}).
  \]
  Since $\tilde{q}$ is a node, $\varphi^{\prime}(\tilde{q})\neq 0$. Thus
  $\min_{i=1}^N \{\mathrm{ord}_{\tilde{q}}(\varphi_i^{\prime}|_{\mathscr
    C^{\prime}_{\zeta,1}})\} = 0$. Hence
  \[
    \min_{i=1}^{N}\{\mathrm{ord}_{q}(\varphi_i^{\prime\prime}|_{\mathscr
      C^{\prime\prime}_{\zeta}}) + c_i \}= A + \deg(\mathscr L^{\prime} \otimes
    (\omega_{\mathscr C^\prime/T}^{\log})^{\otimes A}|_{\mathscr E}).
  \]
  This proves the ``further'' part.
\end{proof}

\begin{lemm}
  \label{lem:contracting-a-chain}
  Given $\xi^\prime$ and suppose $\mathscr B\subset \mathscr C^{\prime}_{\zeta}$
  is a rational bridge. Then there exists a contraction
  $f:\xi^\prime\to \xi^{\prime\prime}$
  that contracts $\mathscr B$
  to a point  
  if and
  only if $\deg (\mathscr L^\prime|_\sB) = 0$.
\end{lemm}

\begin{proof}
  We prove the ``if'' part. Let $f:\mathscr C^\prime \to \mathscr C^{\prime\prime}$
  be the contraction of $\mathscr B$ to some point $q$ and set $\mathscr L^{\prime\prime} =
  f_*\mathscr L^\prime$. We claim that $\mathscr L^{\prime\prime}$ is a line bundle.
  Indeed, since $\mathscr B$ is a rational bridge and $\mathscr L^\prime|_{\mathscr B}$
  is trivial, applying \cite[Corollary~1.5(iii)]{knudsen1983projectivity} we see
  that $f_*\mathscr L^\prime|_{q}$ has rank $1$ for any closed point $q\in \mathscr
  C^{\prime\prime}_{\zeta}$. Since $\mathscr C^{\prime\prime}$  is reduced, the
  claim is proved. Using $f^*\omega_{\mathscr C^{\prime\prime}/T}^{\log}\cong
  \omega_{\mathscr C^{\prime}/T}^{\log}$, we set $\varphi^{\prime\prime} =
  f_*\varphi^\prime$. Since $\varphi^\prime|_{\mathscr B}$ is nonvanishing,
  $\varphi^{\prime\prime}$ is nonzero at the node $q$. Thus we have constructed
  $\xi^{\prime\prime}$ and it is prestable.

  For the ``only if'' part, using the uniqueness part of
  Lemma~\ref{lem:extending-L-and-varphi}, it is easy to see that $\mathscr
  L^\prime \cong f^*\mathscr L^{\prime\prime}$. Thus $\mathscr
  L^{\prime}|_{\mathscr B}$ must be trivial.
\end{proof}

\begin{lemm}
  \label{lem:persistence-of-base-points}
  For any contraction $f:\xi^\prime\to \xi^{\prime\prime}$, if $\xi^\prime$
  violates $\Omega$-2, so does $\xi^{\prime\prime}$.
\end{lemm}
\begin{proof}
  It is straightforward to check that any rational tail or rational bridge
  violating $\Omega$-2 must satisfy $\Omega$-3. Hence by
  Lemmas~\ref{lem:contracting-a-rational-tail} and
  \ref{lem:contracting-a-chain}, contracting a rational tail that violates
  $\Omega$-2 will again violate $\Omega$-2, and a chain of rational bridges satisfying
  $\Omega$-3 cannot be contracted. The contraction $f$ can be decomposed into a
  sequence of contraction of rational tails followed by a sequence of
  contraction of chains of rational bridges. Thus the lemma follows.
\end{proof}

\begin{lemm}
  \label{lem:when-a-rational-tail-is-contracted}
  Suppose $f:\xi^\prime\to \xi^{\prime\prime}$ is a stabilization, then a
  rational tail $\mathscr E\subset \mathscr C^{\prime}_{\zeta}$ is contracted if
  and only if it violates $\Omega$-3.
\end{lemm}
\begin{proof}
  For the ``if'' part, suppose $\mathscr E$ violates $\Omega$-3 but is not
  contracted, we will show that
  $\xi^{\prime\prime}$ cannot be stable. Let $\mathscr C^{\prime}_{\zeta,1} =
  \overline{\mathscr C^\prime_{\zeta} \setminus \mathscr E}$, $\tilde{q} = \mathscr
  C^\prime_{\zeta,1}\cap \mathscr E$ and $q = f(\tilde q)$. If $q$ is a special
  point of $\mathscr C^{\prime\prime}_{\zeta}$, then
  $\varphi^{\prime\prime}(q) \neq 0$. As in the proof of
  Lemma~\ref{lem:contracting-a-rational-tail}, it is easy to see that
  $\deg(\mathscr L^{\prime} \otimes (\omega_{\mathscr
    C^{\prime}/T}^{\log})^{\otimes A}|_{\mathscr E}) = \deg(\mathscr
  L^{\prime\prime} \otimes (\omega_{\mathscr
    C^{\prime\prime}/T}^{\log})^{\otimes A}|_{f(\mathscr E)})$. Hence
  $f(\mathscr E)$ also violates $\Omega$-3. If $q$ is nonspecial, then $f(\mathscr
  C^{\prime}_{\zeta,1}) = q$. We claim that $q$ must violate $\Omega$-2. Indeed,
  since $\xi_{\eta}$ is stable. We have $\deg(\mathscr L^{\prime} \otimes
  (\omega_{\mathscr C^{\prime}/T}^{\log})^{\otimes A}|_{\mathscr C^{\prime}_{\zeta}}) >
  0$. Combining with $\deg(\mathscr L^{\prime} \otimes (\omega_{\mathscr
    C^{\prime}/T}^{\log})^{\otimes A}|_{\mathscr E})< 0$, we have $\deg(\mathscr
  L^{\prime} \otimes (\omega_{\mathscr C^{\prime}/T}^{\log})^{\otimes
    A}|_{\mathscr C^{\prime}_{\zeta,1}})> 0$. Thus, an argument similar to the proof of
  Lemma~\ref{lem:contracting-a-rational-tail} shows that $q$ must violate
  $\Omega$-2.

  For the ``only if'' part, suppose that $\mathscr E$ satisfies $\Omega$-3 and is
  contracted, we show that $\xi^{\prime\prime}$ is unstable. Indeed, let
  $\xi^{\prime\prime\prime}$ be obtained from $\xi^{\prime}$ by contracting
  $\mathscr E$. By Lemma~\ref{lem:contracting-a-rational-tail},
  $\xi^{\prime\prime\prime}$ violates $\Omega$-2.
  Since $f$ contracts $\mathscr E$, by Zariski's main theorem,
  it factors through a contraction $\xi^{\prime\prime\prime}\to
  \xi^{\prime\prime}$.
  By Lemma~\ref{lem:persistence-of-base-points}, $\xi^{\prime\prime}$ violates
  $\Omega$-2.
\end{proof}

\begin{lemm}
  \label{lem:when-a-rational-bridge-is-contracted}
  Suppose $f:\xi^\prime \to \xi^{\prime\prime}$ is a stabilization and suppose
  that every rational tail $\mathscr E\subset \mathscr C^{\prime}_{\zeta}$
  satisfies $\Omega$-3, then a rational bridge $\mathscr B\subset \mathscr
  C^{\prime}_{\zeta}$ is contracted if and only if it violates $\Omega$-3.
\end{lemm}
\begin{proof}
  By Lemma~\ref{lem:when-a-rational-tail-is-contracted}, $f$ contracts no
  rational tails. Hence it only contracts a disjoint union of chains of rational
  bridges. The ``only if'' part follows immediately from
  Lemma~\ref{lem:contracting-a-chain}; the ``if'' part follows from an argument
  similar to the ``if'' part of
  Lemma~\ref{lem:when-a-rational-tail-is-contracted} in the case where $q$ is special.
\end{proof}

\begin{proof}[Proof for Lemma~\ref{lem:stabilization}]
  Given $\xi^\prime$, we
  first use Lemma~\ref{lem:contracting-a-rational-tail} to contract all the
  rational tails in $\mathscr C^{\prime}_{\zeta}$ that violate $\Omega$-3. We repeat this
  process until all rational tails have the property $\Omega$-3. This gives us
  $\xi^{\prime\prime\prime}$. Then we use Lemma~\ref{lem:contracting-a-chain} to
  contract all the rational bridges that violate $\Omega$-3, obtaining
  $\xi^{\prime\prime}$. Then $\xi^{\prime\prime}$ satisfies $\Omega$-2 and all the
  rational tails and rational bridges in $\mathscr C_{\zeta}^{\prime\prime}$
  satisfy $\Omega$-3. Let $\mathscr C^{\prime\prime}_{\zeta,1}\subset \mathscr
  C^{\prime\prime}_{\zeta}$ be any irreducible component that is neither a
  rational bridge or a rational tail. Then either $\omega_{\mathscr
    C^{\prime\prime}/T}^{\log}|_{\mathscr C^{\prime\prime}_{\zeta,1}}$ is
  positive or $\mathscr C^{\prime\prime}_{\zeta,1} = \mathscr
  C^{\prime\prime}_{\zeta}$. In the first case, since
  $\varphi^{\prime\prime}|_{\mathscr C^{\prime\prime}_{\zeta,1}}\neq 0$,
  $\deg(\mathscr L^{\prime\prime} \otimes (\omega^{\log}_{\mathscr
    C^{\prime\prime}/T})^{\otimes c_i} |_{\mathscr C^{\prime\prime}_{\zeta,1}})
  \geq 0$ for some $i$. Since $c_i<A$, we see that $\mathscr
  C^{\prime\prime}_{\zeta,1}$ satisfies $\Omega$-3. In the second case, $\mathscr
  C^{\prime\prime}_{\zeta,1}$ also satisfies $\Omega$-3 because $\xi_{\eta}$ is
  stable. Hence $\xi^{\prime\prime}$ satisfies $\Omega$-3 and is thus stable. This
  proves the existence.

  For the uniqueness, by Lemma~\ref{lem:when-a-rational-tail-is-contracted}, all
  rational tails violating $\Omega$-3 must be contracted.
  Thus any stabilization $\xi^{\prime\prime}$ factors through
  the $\xi^{\prime\prime\prime}$ above.
  By Lemma~\ref{lem:when-a-rational-tail-is-contracted} again, no rational tails
  in $\xi^{\prime\prime}$ can be contracted.
  Thus
  $\xi^{\prime\prime}$ must be obtained from $\xi^{\prime\prime\prime}$
  by contracting chains of rational bridges. Then by
  Lemma~\ref{lem:when-a-rational-bridge-is-contracted}, those rational bridges
  are precisely those that violate $\Omega$-3. Thus the uniqueness follows from
  Lemma~\ref{lem:extending-L-and-varphi}.
\end{proof}

\subsection{Proof for the properness, the projective space case}
\begin{lemm}
  \label{lem:properness-projective-space-case}
  The DM stack $\mathrm{LGQ}^{\Omega}_{g,k}(\mathbb P^{N-1}, d)$
  is proper over $\mathbb C$.
\end{lemm}
\begin{proof}
By the valuative criterion for properness, we need to show
that, possibly after finite base change, any
$\xi_{\eta}\in\mathrm{LGQ}^{\Omega}_{g,k}(\mathbb P^{N-1}, d)(\eta)$ uniquely
extends to
a $\xi\in \mathrm{LGQ}^{\Omega}_{g,k}(\mathbb P^{N-1}, d)(T)$. For this, by a
standard gluing argument, we can assume that $\sC_\eta \to \eta$ is smooth with
connected geometric fiber.

Then the existence of $\xi$ follows from taking the stabilization of the
$\xi^{\prime}$ in Lemma~\ref{lem:prestable-reduction},
using Lemma~\ref{lem:stabilization}.
The ``moreover'' part of Lemma~\ref{lem:prestable-reduction}
implies that any two extensions of $\xi_{\eta}$ are the stabilizations of
a common $\xi^{\prime}$. Thus the uniqueness of $\xi$ follows from the
uniqueness part of Lemma~\ref{lem:stabilization}.
\end{proof}

\section{The general case}

In the previous sections, we have proved
Theorem~\ref{thm:intro-separated-Deligne--Mumford-finite-type} and
Theorem~\ref{thm:intro-properness} for the case of projective spaces
(Lemma~\ref{lem:DM-type-projective-space-case},
Lemma~\ref{lem:properness-projective-space-case}).
In this section, we deal with the general case.

\begin{theo}[Theorem \ref{thm:intro-separated-Deligne--Mumford-finite-type}]
  \label{1}
  The stack $\mathrm{LGQ}^{\Omega}_{g,k}(X,d)$ is a separated Deligne-Mumford
  stack of finite type over $\mathbb C$.
\end{theo}

\begin{theo}[Theorem \ref{thm:intro-properness}]
  \label{2}
  Suppose $V\git_{\theta}G$ is projective, and suppose $S$ is full, then
  $\mathrm{LGQ}^{\Omega}_{g,k}(X, d)$ is proper over $\mathbb C$.
\end{theo}

\subsection{Boundedness}
The proof of boundedness turns out to be quite involved, especially in the
non-abelian case.
We  give a complete proof for the abelian case in this paper, and put
the proof of a key lemma for the non-abelian case in a separate paper
\cite{boundedness-paper}.

The proof should have been a straightforward generalization of the boundedness for
stable quasimaps. However, after discussion with the authors of \cite{ciocan2014stable},
we confirmed that there was a gap in their proof of boundedness.
Indeed, even for one fixed principal bundle, the set of all possible
reductions of the structure group to a given Borel subgroup could be unbounded.
Hence we need to find a new proof from the beginning.
As a byproduct, since stable quaismaps is a special case of LG-quasimaps
(Section~\ref{sec:stable-quasimaps-as-a-special-case}), the result of
\cite{boundedness-paper} also implies the boundedness for stable quasimaps.

\begin{prop}
  The moduli $\mathrm{LGQ}^{\Omega}_{g,k}(X, d)$ is bounded.
\end{prop}
To prove the proposition, we first bound the underlying curves.
\begin{lemm}
  \label{lem:boundedness-underlying-curves}
  The set of twisted curves underlying LG-quasimaps in
  $\mathrm{LGQ}_{g,k}^{\Omega}(X,d)$ is bounded.
\end{lemm}
\begin{proof}
  By Lemma~\ref{lem:bounding-the-order-of-isotropy}, the order of isotropy
  groups at the orbifold points are bounded. In particular, the degree of
  the line bundle $u^*L_{\vartheta} \otimes (\omega_{\mathscr
    C}^{\log})^{\otimes A}$ in Definition~\ref{stability} restricted to any irreducible component
  has bounded denominators. By $\Omega$-3 of
  Definition~\ref{stability}, that line bundle has positive degree on each
  irreducible component. Hence the number of irreducible
  components is also bounded. Hence the underlying curves are
  bounded by Lemma~\ref{lem:bounding-the-order-of-isotropy} and
  \cite[Corollary~1.12]{olsson2007}.
\end{proof}

Now by a standard argument, to prove the proposition it
suffices to show that the
family of underlying principal $\Gamma$-bundles are bounded.
We first explain what this means precisely.
Let $\pi: \Sigma^{\mathscr C} \subset \mathscr C \to S$ be any family of smooth
twisted curves over a finite-type $\mathbb C$-scheme.
For any closed point $s\in S$, let $\mathrm{Bun}_{\Gamma}(\mathscr C_s)$ be the set
of equivalence classes of principal $\Gamma$-bundles on $\mathscr C_s = \pi^{-1}(s)$.
For any subset
\[
  \mathscr S \subset \coprod_{s \in S(\mathbb C)} \mathrm{Bun}_{\Gamma}(\mathscr C_s),
\]
$\mathscr S$ is said to be bounded if there exists a finite-type $S$-scheme
$T$, a principal $\Gamma$-bundle $\mathscr P$ on $\mathscr C\times_{S} T$,
such that for any $\mathscr P_s \in \mathscr S \cap \mathrm{Bun}_{\Gamma}(\mathscr
C_s)$, there exists $t\in T(\mathbb C)$, lying over $s$, such that $\mathscr
P|_{t}$ is isomorphic to $\mathscr P_s$.

To bound the underlying principal bundles, we first bound some topological
invariants. Suppose $\mathscr C$ is a twisted curve and $\mathscr P$ is a principal
$\Gamma$-bundle on $\mathscr C$. Then its degree, denoted by $\deg_{\mathscr
  P} \in \mathrm{Hom}(\widehat{\Gamma}, \mathbb Q)$, is defined to be
\[
  \deg_{\mathscr P} (\alpha) = \deg(\mathscr P\times_{\Gamma} \mathbb
  C_{\alpha}), \quad \alpha \in \widehat{\Gamma}.
\]
More generally, if $\mathscr C^\prime \subset \mathscr C$ is an irreducible
component, then we set
\[
  \deg_{\mathscr P, \mathscr C^\prime} (\alpha) = \deg(\mathscr P\times_{\Gamma} \mathbb
  C_{\alpha} |_{\mathscr C^\prime}), \quad \alpha \in \widehat{\Gamma}.
\]
If $\mathscr P$ is the underlying principal $\Gamma$-bundle of some
LG-quasimap $\xi$, we sometimes also write $\xi$ in place of $\mathscr P$ in the subscript.

\begin{lemm}
  \label{lem:bounding-epsilon-and-vartheta}
  Suppose $\xi \in \mathrm{LGQ}_{g,k}^{\Omega}(X,d)(\mathbb
  C)$ and $\mathscr C^\prime$ is an irreducible component of the domain curve.
  There exists $B \in \mathbb R$, independent of $\xi$ and $\mathscr C^\prime$, such that
  \[
    |\deg_{\xi,
      \mathscr C^\prime} (\epsilon) | < B  \quad \text{and} \quad
    | \deg_{\xi, \mathscr C^\prime} (\vartheta)|  < B.
  \]
\end{lemm}
\begin{proof}
  Since the family of
  underlying curves is bounded, $\deg_{\xi,
    \mathscr C^\prime} (\epsilon) = \deg(\omega^{\log}_{\mathscr C}|_{\mathscr
    C^\prime})$ is bounded.
  By the $\Omega$-stability condition, $u^* L_{\vartheta} \otimes (\omega_{\mathscr
    C}^{\log})^{\otimes A}$ has non-negative degree on each irreducible
  component of $\mathscr C$. Since its total degree is fixed, $\deg(u^*L_\vartheta \otimes
  (\omega_{\mathscr C}^{\log})^{\otimes A}|_{\mathscr C^\prime})$ is bounded.
  Hence $\deg_{\xi, \mathscr C^\prime} (\vartheta) = \deg(u^*L_\vartheta
  |_{\mathscr C^\prime})$
  is also bounded.
\end{proof}

By the above discussion, we have reduced the proposition to the following
lemma.
\begin{lemm}
  \label{lem:boundedness-key}
  Let $\pi: \Sigma^{\mathscr C} \subset \mathscr C \to S$ be a family of
  twisted curves over a finite-type $\mathbb C$-scheme.
  Fix any $d_1,d_2 \in \mathbb Z$,
  the set of principal $\Gamma$-bundles $\mathscr P_s$ on $\mathscr C_s$, $s\in
  S(\mathbb C)$, such that
  \begin{itemize}
  \item
    $\mathscr P_s
    \times_{\Gamma} V \to \mathscr C_s$ admits a section $\sigma$ sending
    the generic points of $\mathscr C_s$ into $\mathscr P_s \times_{\Gamma} V^{\mathrm{s}}(\theta)$,
  \item
    for each irreducible component $\mathscr C^\prime \subset \mathscr C_{s}$,
    we have
    \[
      \deg_{\mathscr P_s, \mathscr C^\prime}(\vartheta)  = d_1 \quad \text{and} \quad
      \deg_{\mathscr P_{s}, \mathscr C^\prime}(\epsilon)  = d_2,
    \]
  \end{itemize}
  is bounded.
\end{lemm}
We only prove this lemma in the abelian case here. The general case is the main
theorem in \cite{boundedness-paper}.

From now on assume that $G = (\mathbb C^*)^{\times n}, \Gamma = G \times
\mathbb C^*$, and $\epsilon$ is the second projection.
Thus a principal $\Gamma$-bundle is equivalent to an $(n+1)$-tuple of
line bundles as in Section~\ref{sec:brief}. Hence to bound the underlying
$\Gamma$-bundles, it suffices to bound the degree of each line bundle on each
$\mathscr C^\prime$.
Once we have done that, the boundedness of those line
bundles follows easily from the boundedness of the relative Picard functor for
smooth projective curves, combined with a standard splitting-node argument,
and a standard descent argument using the following covering trick.
\begin{lemm}
  \label{lem:covering-trick}
  Let $\pi: \Sigma^{\mathscr C} \subset \mathscr C \to S$ be a family of smooth
  twisted curves over a finite-type $\mathbb C$-scheme $S$.
  Then \'etale locally on $S$, there exists a family of connected smooth
  projective curves $\tilde{\mathscr C} \to S$, together with a faithfully flat
  $S$-morphism $\tilde{\mathscr C} \to \mathscr C$.
\end{lemm}
\begin{proof}
  Let $r$ be the least
  common multiple of the orders of automorphism groups of points in $\Sigma$.
  Let $\rho: \mathscr C \to \underline{ \mathscr C}$ be the coarse moduli and
  set $D  = \rho(\Sigma^{\mathscr C})$.
  Then possibly after base change to an \'etale cover of $S$,
  there exists a relative effective Cartier divisor $H \subset
  \underline{\mathscr C}$ such that $D+H$ is \'etale over $S$ and $\mathcal
  O_{\underline{\mathscr C}}(D + H) \cong \mathscr L^{\otimes r}$ for some
  line bundle $\mathscr L$ on $\underline{\mathscr C}$. Using that we
  construct the
  cyclic cover
  $\tilde{\mathscr C} \to \underline{\mathscr C}$, which is
  totally ramified
  over $D + H$ with ramification index equal to $r - 1$. Hence it lifts to
  a map $\tilde{\mathscr C} \to \mathscr C$ by \cite[Theorem~4.2.1]{abramovich2008gromov}.
  This finishes the construction.
\end{proof}

Now we come to bounding the degrees of the line bundles.
We may assume $\vartheta$ is the pullback of $\theta$, thanks to
Lemma~\ref{lem:dependence-on-C-and-vartheta}.
We choose $\theta_i \in \widehat{G}$, $i=1 ,\ldots, n$, such that
\begin{itemize}
\item $V^{\mathrm{s}}(\theta) \subset V^{\mathrm{ss}}(\theta_i)$,
\item $\theta \in \sum_{i} \mathbb R_{> 0} \theta_i$,
\item $\theta_1 ,\ldots, \theta_n$ span $\widehat{G} \otimes \mathbb R$.
\end{itemize}
This is possible because any $\theta_i$ such that the ray $\mathbb R_{\geq
  0}\theta_i$ is sufficiently close to $\mathbb R_{\geq 0}\theta$ satisfies the
first requirement.

Let $\vartheta_i$ be the pullback of $\theta_i$.
Let $\mathscr P_s$, $\mathscr C_s$ and $\mathscr
C^\prime$ be as in Lemma~\ref{lem:boundedness-key}.
\begin{lemm}
  \label{lem:bounding-theta_i}
  There exists $B \in \mathbb R$, independent of $s$,
  $\mathscr P_s$ and $\mathscr C^\prime$, such that
  \[
    \deg_{\mathscr P_s, \mathscr C^\prime}(\vartheta_i) >  B, \quad i = 1
    ,\ldots, n.
  \]
\end{lemm}
\begin{proof}
  As in Lemma~\ref{lem:second-decomposition}, for each $k\geq 0$,
  we have
  \[
    \textstyle
    \Gamma(V, L_{k\theta_i})^{G} = \bigoplus_{c\in \mathbb Z} \Gamma(V,
    L_{k\vartheta_i + c\epsilon})^{\Gamma}.
  \]
  Let $x\in V^{\mathrm{s}}(\theta) \subset
  V^{\mathrm{ss}}(\theta_i)$ be a point in the image of $u|_{\mathscr
    C^\prime}$, where $u: \mathscr C_s \to [V/\Gamma]$ is the map defined by
  $\mathscr P_s$ and $\sigma$.
  We can choose an
  $f\in \Gamma(V, L_{k\vartheta_i + c\epsilon})^{\Gamma}$, for some $k > 0$ and
  $c\in \mathbb Z$, such that $f(x) \neq 0$. Viewing $f$ as a $\Gamma$-equivariant
  map $f:V \to \mathbb C_{k\vartheta_i + c\epsilon}$, it induces a map
  \[
    \mathscr P_s\times_{\Gamma} V \longrightarrow \mathscr P_s\times_{\Gamma} \mathbb
    C_{k\vartheta_i + c\epsilon},
  \]
  under which $\sigma$ is mapped to a  section that is nonzero at the generic
  point of $\mathscr C^\prime$. Hence
  \[
    \deg(\mathscr P_s\times_{\Gamma} \mathbb
    C_{k\vartheta_i + c\epsilon} |_{\mathscr C^\prime}) \geq 0.
  \]
  Writing $\mathscr L_{\vartheta_i} = \mathscr P_{s}\times_{\Gamma} \mathbb
  C_{\vartheta_i}$ and $\mathscr L_{\epsilon} = \mathscr P_{s}\times_{\Gamma} \mathbb
  C_{\epsilon}$, we have
  \[
    \mathscr P_s\times_{\Gamma} \mathbb C_{k\vartheta_i + c\epsilon}  \cong
    \mathscr
    L_{\vartheta_i}^{\otimes k} \otimes \mathscr L_\epsilon^{ \otimes c}.
  \]
  Thus
  \[
    \deg(\mathscr L_{\vartheta_i}|_{\mathscr C^\prime})
    \geq
    -\frac{cd_2}{k}.
  \]
  It is easy to see that the $f$ above can be taken from a finite set, since
  the ideal of the unstable locus is finitely generated.
  Hence we get a uniform bound.
\end{proof}

Now we have established enough bounds for the proof.
\begin{proof}[Proof of Lemma~\ref{lem:boundedness-key} in the abelian case]
  Pick a common $B$ for Lemma~\ref{lem:bounding-epsilon-and-vartheta} and
  Lemma~\ref{lem:bounding-theta_i}.
  Define
  \begin{equation*}
    \Omega = \Bigg\{d \in \mathrm{Hom}(\widehat {\Gamma}, \mathbb R) \Bigg |~
    \begin{aligned}
      & d(\vartheta) = d_1, \\
      & d(\epsilon) = d_2, \\
      & d( - \vartheta_i) < B, i = 1 ,\ldots,  n~
    \end{aligned}
    \Bigg\}.
  \end{equation*}
  Since the origin of $\widehat\Gamma$ is in the interior of the convex hull of $\pm
  \epsilon, \vartheta, - \vartheta_1 ,\ldots, -\vartheta_n$, $\Omega$ is bounded.
  We have seen that any $\deg_{\xi, \mathscr C^\prime}$ of
  Lemma~\ref{lem:boundedness-key} lies in $\Omega$.
  Thus the degrees of each line bundle associated to $\mathscr P_s|_{\mathscr
    C^\prime}$ is uniformly bounded. This completes the proof.
\end{proof}

\subsection{Further reductions}
\label{sec:quick-reduction-general-case}
Assuming the boundedness,
the next goal is to prove Theorem~\ref{1} and Theorem~\ref{2} using the
valuative criterion.
We first make some quick reductions.
We claim that without loss of generality we can make the following assumptions:
\begin{enumerate}
\item
  $V$ is reduced.
\item
  the affine quotient $V\git_0G = \operatorname{Spec} R_{0}$ is connected.
\item
  $S\subset R_{\theta}$.
\item
  When $V\git_{\theta}G$ is projective and $S$ is full,
  we assume in addition that $R_{\bullet}$ is generated by $S$ as a $\mathbb C$-algebra.
\end{enumerate}
Note that
when $V\git_{\theta}G$ is projective, (1) and (2) imply that $R_{0} = \mathbb C$.

We now show that it indeed suffices to prove the two theorems with these additional assumptions.
When proving the theorems, we consider the valuative criterion situation. Since the families
of curves appearing in the valuative criterions are reduced and connected DM
stacks, (1) and (2) can be assumed.
By raising each $f\in S$ to some power, which does not change the
$\Omega$-stability, we may assume that $S\subset
R_{m\theta}$ for some $m>0$. Then replacing $\vartheta$ by $m\vartheta$ and $A$ by
$mA$, we get (3).

Now suppose $V\git_{\theta} G$ is projective and $S$ is full.
Since we have assumed $V$ is reduced, by
Corollary~\ref{cor:projective-case-independent-of-S}, we may re-arrange any full
$S$ without changing the $\Omega$-stability.
There exists $m>0$ such that $\bigoplus_{m=0}^\infty R_{m\theta}$
is generated by $R_{m\theta}$ as an ($R_{0}\cong \mathbb C$)-algebra (c.f.\
\cite[Exercise~7.4.G]{vakil2025rising}).
Replacing
$\vartheta$ by $m\vartheta$, and $A$ by $mA$, we may assume
that $m=1$. Take $S$ to be a homogeneous basis for $R_{\theta}$ and we get (4).

\subsection{Comparing LG-quasimaps to $X$ and LG-quasimaps to $\mathbb P^{N-1}$}
\label{sec:comparing-LG-to-X-and-to-P-N-1}

Having made the additional assumptions in the previous section,
now the $\Omega$-stability for LG-quasimaps to $X$ is defined by the triple
$\Omega=(S,A,\vartheta)$, where
\[
  S = \{f_1 ,\ldots, f_N\}, \quad f_i \in R_{\vartheta, c_i\epsilon}, c_i\in \mathbb Z,\, i=1 ,\ldots, N.
\]
Recall the $R$-charged package for $\mathbb P^{N-1}$ from
Section~\ref{sec:projective-space-example}, with the $c_i$'s the same as above.
We will compare $\Omega$-stable LG-quasimaps to $X$ and $\Omega^\prime$-stable
LG-quasimaps to $\mathbb P^{N-1}$, for
$\Omega^\prime=(S^\prime,A^\prime,\vartheta^\prime)$, where
$S'=(x_1,\cdots,x_N)$ consisting of the standard coordinates
of $\CC^N$, $A'=A$ and $\vartheta^\prime = \mathrm{pr}_1$ is the projection onto
the first factor.

To compare the two $R$-charged packages for $X$ and $\mathbb P^{N-1}$
respectively, consider the map
\begin{equation}
  \label{map:f}
  f = (f_1 ,\ldots, f_N): V \longrightarrow \mathbb C^{N},
\end{equation}
and the group homomorphism
\[
  \phi = (\vartheta, \epsilon): \Gamma \longrightarrow (\mathbb C^*)^{\times 2}.
\]
Then the two $R$-charged package are compatible in that $f$ is equivariant with
respect to $\phi$, and $\phi^{*}\mathrm{pr}_2 = \epsilon$.
Their stability conditions are also compatible in the sense that $f\sta S'=S$ and $\phi^*
\vartheta^\prime = \vartheta$.

Let
\[
  \mathrm{LGQ}^{\Omega1}_{g,k}(X, d) \subset \mathrm{LGQ}^{\mathrm{pre}}_{g,k}(X, d)
\]
be the open substack consisting of LG-quasimaps satisfying $\Omega$-1. By the
compatibility above,
$(f,\phi)$ induces a morphism of stacks
\begin{equation}
  \label{map:f-star}
  \Phi : \mathrm{LGQ}_{g,k}^{\Omega1}(X, d) \longrightarrow
  \mathrm{LGQ}^{\mathrm{pre}}_{g,k}(\mathbb P^{N-1}, d).
\end{equation}
For any scheme $T$, it sends
\[
  (\mathscr C, \Sigma^{\mathscr C}, u, \kappa) \in
  \mathrm{LGQ}_{g,k}^{\Omega1}(X, d)(T)
\]
to
\[
  (\underline{\mathscr C}, \Sigma^{\underline{\mathscr C}},
  \mathscr L, \varphi_1 ,\ldots, \varphi_N) \in
  \mathrm{LGQ}_{g,k}^{\mathrm{pre}}(\mathbb P^{N-1},d)(T),
\]
where $\rho: (\mathscr C, \Sigma^{\mathscr C}) \to
(\underline{\mathscr C}, \Sigma^{\underline{\mathscr C}})$ is the map to the coarse
moduli space, $\mathscr L = \rho_*(u^*L_{\vartheta})$, and $\varphi_i =
\rho_*(u^*f_i)$, $i=1 ,\ldots, N$. Here we are using the description of the objects of
$\mathrm{LGQ}_{g,k}^{\mathrm{pre}}(\mathbb P^{N-1},d )$ in terms of line bundles
and sections as in Section~\ref{sec:projective-space-example}.
Note that
$\omega_{\mathscr C/T}^{\log}$ is canonically isomorphic to
$\rho^*(\omega_{\underline{\mathscr C}/T}^{\log})$. Hence indeed
$\varphi_i =\rho_*(u^*f_i)\in H^0(\underline{\mathscr C}, \mathscr L \otimes
(\omega^{\log}_{\underline{\mathscr C}/T})^{\otimes c_i})$.
Also, we need to justify that $\deg(\mathscr L) = \deg(u^*L_{\vartheta})$. Indeed,
at each orbifold point, which must be a node or marking, at least one of $u^*f_i \in H^0(\mathscr C,
u^*L_{\vartheta})$ is nonzero. Hence $u^*L_{\vartheta}$ does not have any monodromy at
the orbifold points of $\mathscr C$. Hence we have $\rho^*\mathscr L\cong
u^*L_{\vartheta}$ and $\deg(\mathscr L) = \deg(u^*L_{\vartheta})$.
\begin{lemm}
  \label{lem:stable-locus-coincides}
  We have $\Phi^{-1}(\mathrm{LGQ}^{\Omega}_{g,k}(\mathbb P^{N-1},d)) = \mathrm{LGQ}^{\Omega}_{g,k}(X,d)$.
\end{lemm}
\begin{proof}
  This is straightforward.
\end{proof}

\subsection{The local valuative criterion}
\label{sec:valuative-criterion}
Combining  Lemma~\ref{lem:properness-projective-space-case} and
Lemma~\ref{lem:stable-locus-coincides}, to prove that
$\mathrm{LGQ}^{\Omega}_{g,k}(X,d)$ is separated (resp.~proper),
it suffices to prove that $\Phi$ is separated (resp.~proper).

We use the valuative criterion for $\Phi$. This can be checked locally on
$\underline{\mathscr C}$, which we now explain.
Let $T$ be the spectrum of a
discrete valuation ring with residue field $\CC$ and fractional field $K$.
We denote by $\eta$ and $\zeta$ the generic and closed points of $T$,
respectively. Given
\[
  \underline{\xi} = (\underline{\mathscr C}, \Sigma^{\underline{\mathscr C}},
  \mathscr L, \varphi_1 ,\ldots, \varphi_N) \in
  \mathrm{LGQ}_{g,k}^{\mathrm{pre}}(\mathbb P^{N-1},d)(T),
\]
\[
  \xi_{\eta} = (\mathscr C_{\eta}, \Sigma^{\mathscr C_\eta}, u_{\eta}, \kappa_{\eta}) \in
  \mathrm{LGQ}_{g,k}^{\Omega1}( X,d)(\eta),
\]
and an isomorphism
\[
  \alpha_{\eta}: \Phi(\xi_{\eta}) \overset{\cong}{\longrightarrow}
  \underline{\xi}_{\eta} := \underline{\xi}|_{\eta},
\]
possibly after finite base change,
we need to investigate the existence and uniqueness
of pairs $(\xi,\alpha)$, where
$\xi\in \mathrm{LGQ}_{g,k}^{\Omega1}(
X,d)(T)$ is an extension of $\xi_{\eta}$, and $\alpha: \Phi(\xi) \to
\underline{\xi}$ is an isomorphism extending $\alpha_{\eta}$.
We refer to  this as the valuative criterion for $(\underline{\xi}, \xi_{\eta},
\alpha_{\eta})$.

We now work locally on $\underline{\mathscr C}$. Obviously, one can extend the definition of
$\Phi$ to nonproper family of curves. The definition is verbatim and we omit the details.
Using that, the valuative criterion situation for $\Phi$ can be stated for
nonproper family of curves without any change.
For any Zariski open $\underline{\mathscr U} \subset \underline{\mathscr C}$, we
base change various objects to $\underline{\mathscr U}$, denoted as
\[
  \mathscr U_{\eta} = \mathscr C_{\eta}|_{\underline{\mathscr U}}:= \mathscr C_{\eta}\times_{\underline{\mathscr C}}
  \underline{\mathscr U},
  \quad \underline{\xi}|_{\underline{\mathscr U}}:= (\underline{\mathscr U},
  \underline{\mathscr U} \cap \Sigma^{\underline{\sC}}, \mathscr
  L|_{\underline{\mathscr U}}, \varphi|_{\underline{\mathscr U}}),
  \quad {\xi}_{\eta}|_{\underline{\mathscr U}} := \xi_{\eta}|_{\mathscr U_{\eta}}
\]
and so on.

\begin{defi}[The local valuative criterion]
  We say that $\Phi$ satisfies the local valuative criterion if for any
  $(\underline{\xi}, \xi_{\eta}, \alpha_{\eta})$ as above, there is a
  basis $\{\underline{\mathscr U}_\lambda\}_{\lambda\in \Lambda}$ for the Zariski topology on
  $\underline{\mathscr C}$, such that for each $\lambda \in \Lambda$,
  $(\underline{\xi}|_{\underline{\mathscr U}_\lambda}, \xi_{\eta}|_{\underline{\mathscr U}_\lambda},
  \alpha_{\eta}|_{\underline{\mathscr U}_\lambda})$ satisfies the valuative criterion.
\end{defi}

\begin{lemm}
  If $\Phi$ satisfies the local valuative criterion for separatedness (resp.
  separatedness and properness) , then $\Phi$ is separated (resp. proper).
\end{lemm}
\begin{proof} This follows from the standard gluing argument, and will be omitted.
\end{proof}

\subsection{Checking the local valuative criterion}
Let $T,\underline{\xi}, \xi_{\eta}, \alpha_{\eta}$ be as before.
Let $\underline{\mathscr U}\subset \underline{\mathscr C}$ be any Zariski open
such that $\omega_{\underline{\mathscr C}/T}^{\log}|_{\underline{\mathscr U}}$ is
trivial. We will prove that $( \underline{\xi}|_{\underline{\mathscr U}},
\xi_{\eta}|_{\underline{\mathscr U}}, \alpha_{\eta}|_{\underline{\mathscr U}})$
satisfies the valuative criterion. Once this is done, since all such $\underline{\mathscr U}$
form a basis for the Zariski topology on $\underline{\mathscr C}$, this will verify the local
valuative criterion.

From now on we add the assumption that
$\underline{\mathscr U}$ contains no base points of $\underline{\xi}$. The
general case then follows from a standard argument using the algebraic Hartogs'
theorem (c.f.\ \cite[Lemma~4.3.2]{ciocan2014stable}).

The idea of the proof is to locally compare maps into $X$ and
LG-quasimaps to $X$.
We write $V_{f} = V\setminus f^{-1}(0) \subset V^{\mathrm{ss}}(\theta)$ and $X_f = [V_f/G]\subset X$.
Let $y:\mathscr Y\to \underline{\mathscr U}$ be any separated DM stack over
$\underline{\mathscr U}$ and let $\underline{\mathscr Y}$ be its coarse moduli.
Invoking the definition of LG-quasimaps in
Definition~\ref{def:prestable-LG-quasimap}, we consider the following categories
\begin{itemize}
\item
  $\mathrm{LG}^{\mathrm{bf}}(\mathscr Y, X_f)$ is the category of pairs $(u_{y}, \kappa_y)$,
  where $u_{y}: \mathscr Y \to [V_f/\Gamma]$ is
  representable and $\kappa_y:
  u_y^*L_{\epsilon} \to y^* (\omega_{\underline{\mathscr
      C}}^{\log}|_{\underline{\mathscr U}})$
  is an isomorphism of line bundles.
\item
  $\mathrm{LG}^{\mathrm{bf}}(\underline{\mathscr Y}, \mathbb P^{N-1})$ is the category of
  pairs $(u_{y}, \kappa_y)$, where $u_{y}: \underline{\mathscr Y} \to
  [\mathbb C^{N}\setminus \{0\}/(\mathbb C^*)^{\times 2}]$ and $\kappa_y:
  u_y^*L_{\mathrm{pr}_2} \to y^* (\omega_{\underline{\mathscr
      C}}^{\log}|_{\underline{\mathscr U}})$
  is an isomorphism of line bundles.
\item
  $\mathrm{Hom}^{\mathrm{rep}}(\mathscr Y, X_f)$ is the category of
  representable morphisms $\mathscr Y\to X_f$.
\item
  $\mathrm{Hom}(\underline{\mathscr Y}, \mathbb P^{N-1})$ is the set of
  morphisms $\underline{\mathscr Y}\to \mathbb P^{N-1}$.
\end{itemize}
Here the upper script ``bf'' means ``base point free''.
Note that \eqref{map:f} is equivariant with respect to $\theta: G\to \mathbb
C^*$. Hence it descends to
\[
  F : X_{f} \longrightarrow \mathbb P^{N-1}.
\]

\begin{lemm} \label{lem:map-vs-LG}
  Suppose $\omega_{\underline{\mathscr C}}^{\log}|_{\underline{\mathscr
      U}}\cong \sO_{{\underline{\mathscr U}}}$, there exists a commutative
  diagram
  \begin{equation}
    \begin{tikzcd}
      \mathrm{LG}^{\mathrm{bf}}(\mathscr Y, X_f) \arrow[r] \arrow[d] &
      \mathrm{LG}^{\mathrm{bf}}(\underline{\mathscr Y}, \mathbb P^{N-1}) \arrow[d] \\
      \mathrm{Hom}^{\mathrm{rep}}(\mathscr Y, X_f) \arrow[r] &
      \mathrm{Hom}(\underline{\mathscr Y}, \mathbb P^{N-1})
    \end{tikzcd},
  \end{equation}
  where the upper (resp.\ lower) horizontal arrow is defined by composition with $f: V \to \mathbb C^{N}$
  (resp.\ $F$) in the same way as $\Phi$ is defined
  (c.f.~Section~\ref{sec:comparing-LG-to-X-and-to-P-N-1}),
  and the two vertical arrows are equivalence of categories. Moreover, the diagram is natural in
  $\mathscr Y$ and compatible with base change for $T$.
\end{lemm}
\begin{proof}
  From the exact sequence of groups
  \[
    1 \longrightarrow G \longrightarrow \Gamma \longrightarrow \mathbb C^*
    \longrightarrow 1,
  \]
  we obtain a 2-categorical fibered diagram
  \[
    \begin{tikzcd}
      {[V_f / G]} \ar[r] \ar[d] & \operatorname{Spec} (\mathbb C) \ar[d]\\
      {[V_f / \Gamma]} \ar[r, "\epsilon_*"] & B \mathbb C^*
    \end{tikzcd}.
  \]
  For $u_y: \mathscr Y \to [V_f/\Gamma]$, $\epsilon_*\circ u_y$ is the
  classifying morphism for the frame bundle of $u^*_y L_{\epsilon}$.
  We fix a trivialization for $\omega_{\underline{\mathscr C}}^{\log}|_{\underline{\mathscr U}}$ once and for all.
  Composed with this trivialization,
  the $\kappa_y$ in the definition of $\mathrm{LG}^{\mathrm{{b}f}}(\mathscr Y, X_f)$
  becomes a trivialization of $u^*_y L_{\epsilon}$, which is equivalent to a
  lifting of $\epsilon_* \circ u_y$ to $\operatorname{Spec} (\mathbb C)$.
  Thus, by the fibered diagram, $(u_y, \kappa_y)$ is equivalent to a
  morphism $\mathscr Y \to X = [V_f / G]$. Since $[V_f/G] \to [V_f /
  \Gamma]$ is representable, a morphisms $\mathscr Y \to [V_f / \Gamma]$ is
  representable if and
  only if the composed morphism $\mathscr Y \to [V_f/\Gamma]$ is.

  This establishes the equivalence of categories for the left vertical arrow.
  The case for the right vertical arrow is similar, and the rest is straightforward.
\end{proof}

Let
\[
  \underline{h}: \underline{\mathscr U} \longrightarrow \mathbb P^{N-1} \quad \text{and}
  \quad
  h_{\eta}: \mathscr U_{\eta} \longrightarrow X_f
\]
be obtained from $\underline{\xi}|_{\underline{\mathscr U}}$ and
$\xi_{\eta}|_{\underline{\mathscr U}}$ respectively by applying the vertical arrows in
Lemma~\ref{lem:map-vs-LG}. Then using the commutativity and naturality in
Lemma~\ref{lem:map-vs-LG},
$\alpha_{\eta}|_{\underline{\mathscr
    U}}$ reads that
$\underline{h}|_{\eta} = \underline{h_\eta}$, where
$\underline{h_\eta}:\underline{\mathscr U}|_{\eta} \to \mathbb P^{N-1}$ is the
map induced by the composition of $\mathscr U_{\eta}\overset{h_{\eta}}{\to} X_f \overset{F}{\to} \mathbb P^{N-1}$.

Applying Lemma~\ref{lem:map-vs-LG}, the valuative criterion for $( \underline{\xi}|_{\underline{\mathscr U}},
\xi_{\eta}|_{\underline{\mathscr U}}, \alpha_{\eta}|_{\underline{\mathscr
    U}})$ translates into the following
\begin{lemm}
  Up to isomorphisms there exists at most one
  $(\mathscr U, \Sigma^{\mathscr U}, h)$, such that
  \begin{enumerate}
  \item
    $\Sigma^{\mathscr U} \subset \mathscr U \to T$ is (nonproper) family of
    pointed balanced twisted curves and $h: \mathscr U\to X_f$ is a
    representable morphism, extending $(\mathscr U_{\eta}, \Sigma^{\mathscr
      U_{\eta}}:= \mathscr U_{\eta}\cap \Sigma^{\mathscr C_{\eta}}, h_{\eta})$,
  \item
    the coarse moduli of $\mathscr U$ is isomorphic to $\underline{\mathscr
      U}$, extending the identification of the coarse moduli of $\mathscr
    U_{\eta}$ and $\underline{\mathscr U}|_{\eta}$ coming from
    $\alpha_{\eta}|_{\underline{\mathscr U}}$,
  \item
    the map $\underline{\mathscr U} \to \mathbb P^{N-1}$ induced by the
    composition of $\mathscr U \overset{h}{\to} X_f \overset{F}{\to} \mathbb
    P^{N-1}$ is equal to $\underline{h}$.
  \end{enumerate}
  When $\underline{X} = V\git_{\theta} G$ is projective and $S$ is full,
  possibly after finite base change
  such a triple $(\mathscr U, \Sigma^{\mathscr U}, h)$ always exists.
\end{lemm}
\begin{proof}
  First note that once (1) and (2) are satisfied, (3) is automatic since
  $\underline{\mathscr U}|_{\eta}$ is dense in $\underline{\mathscr U}$ and
  $\underline{\mathscr U}$ is reduced.
  Let $\underline{X}_f$ be the coarse moduli of $X_{f}$.
  Let $\underline{h}^{\prime}_{\eta}: \underline{\mathscr U}|_{\eta} \to
  \underline{X}_f$ be the map induced by $h_{\eta}$, and let $\underline{F}:
  \underline{X}_{f} \to \mathbb P^{N-1}$ be the map induced by $F$.

  Given any $(\mathscr U, \Sigma^{\mathscr U}, h)$ satisfying (1) and (2), $h$ induces an
  \[
    \underline{h}^\prime: \underline{\mathscr U} \longrightarrow \underline{X}_f,
  \]
  extending $\underline{h}^{\prime}_{\eta}$.
  Conversely, given such an $\underline{h}^\prime$ extending
  $\underline{h}^{\prime}_{\eta}$, possibly after finite base change there
  exists a unique $(\mathscr U, \Sigma^{\mathscr U}, h)$ up to unique
  isomorphism satisfying (1) and (2)
  such that $\underline{h}^\prime$ is the induced map between coarse moduli spaces.
  Indeed, this follows from the proof for
  \cite[Proposition~6.0.1]{abramovich2002compactifying}. More precisely, the
  $f_1: C_1 \to \mathbf M$ in \cite{abramovich2002compactifying} constructed in Step 1 is replaced by the given
  $(\underline{\mathscr U}, \Sigma^{\underline{\mathscr U}},
  \underline{h}^\prime)$.
  The remaining steps are constructions that are local on the coarse moduli of the curves.

  To finish the proof, it remains to study the existence and uniqueness of
  the $\underline{h}^\prime$ extending $\underline{h}^{\prime}_{\eta}$.
  Since $\underline{\mathscr U}|_{\eta}$ is dense in $\underline{\mathscr U}$,
  $\underline{\mathscr U}$ is reduced,
  and $\underline{X}_f$ is separated,
  $\underline{h}^\prime$ is unique if it exists.
  When $\underline{X}$ is projective and $S$ is full, $\underline{X}_f =
  \underline{X}$ and $\underline{F}$ is a closed embedding.
  By continuity the given map $\underline{\mathscr U} \to \mathbb P^{N-1}$
  factors through $\underline{X}_f$. Hence in this case $\underline{h}^\prime$ always exists.
  This completes the proof.
\end{proof}

\appendix
\section{Comparison with old stability condition}
\label{sec:appendix}
In this appendix, we compare the $\Omega$-stability condition with the stability
condition in \cite{chang2019mixed} for the Fermat quintic.
Recall the setup of prestable MSP fields from Section~\ref{sec:brief}.
\begin{lemm}
  \label{lem:msp-finite-auto-and-positivity}
  Let $\frac{1}{5} < A < \frac{2}{5}$.
  Let
  \begin{equation}
    \label{eq:non-vanishing-appendix}
    \xi =
    (\mathscr C, \Sigma^{\mathscr C}, \mathscr L, \mathscr N, \varphi, \rho, \mu,
    \nu)
  \end{equation}
  be a prestable MSP field such that
  \[
    (\varphi, \mu) \neq 0, \quad (\mu, \nu) \neq 0, \quad (\rho, \nu)\neq 0
  \]
  hold everywhere on $\mathscr C$. Then the automorphism group of $\xi$ is finite if
  and only if the $\mathbb Q$-line bundle
  $\mathscr L \otimes \mathscr N^{\otimes 2} \otimes (\omega_{\mathscr
    C}^{\log})^{\otimes A}$ has positive degree on each irreducible component of
  $\mathscr C$.
\end{lemm}
\begin{proof}
  By Theorem~\ref{thm:intro-separated-Deligne--Mumford-finite-type}, (2) implies
  that that $\xi$ has finitely many automorphisms.

  Conversely, suppose that $\mathscr L \otimes \mathscr N^{\otimes 2}
  \otimes (\omega_{\mathscr C}^{\log})^{\otimes A}$ has non-positive degree on
  some irreducible component.
  We would like to show that $\mathrm{Aut}(\xi)$ is infinite.
  Obviously it suffices to consider the case when $\mathscr C$ is smooth and
  irreducible and  $\mathrm{Aut}(\mathscr C, \Sigma^{\mathscr C})$ is infinite.
  Thus from now on we assume
  \begin{equation}
    \label{eq:negative-degree-appendix}
    \deg(
    \mathscr L \otimes \mathscr N^{\otimes 2}
    \otimes (\omega_{\mathscr C}^{\log})^{\otimes A}) \leq  0.
  \end{equation}

  We make several observations.
  \begin{enumerate}[(i)]
  \item
    $\deg(\mathscr L \otimes \mathscr N^{\otimes 2} \otimes (\omega_{\mathscr
      C}^{\log})^{\otimes (1/5)}) \leq 0$. Indeed $\deg(\omega_{\mathscr
      C}^{\log})\in \mathbb Z$ and \eqref{eq:non-vanishing-appendix} implies
    $\deg(\mathscr L), \deg(\mathscr N) \in \frac{1}{5} \mathbb Z$.
  \item
    $\deg(\mathscr L \otimes \mathscr N) \in \mathbb Z$. Indeed,
    suppose $q\in \mathscr C$ is an orbifold point. If $\mathscr L \otimes
    \mathscr N$ has nontrivial monodromy at $q$. Then $\mu(q) = 0$. Then
    $\varphi(q) \neq 0$ and $\nu(q) \neq 0$. But this implies that both $\mathscr
    L$ and $\mathscr N$ have nontrivial monodromy at $q$, a contradiction.
    Hence $\mathscr L \otimes \mathscr N$ has trivial monodromy at every orbifold
    point of $\mathscr C$. Thus its degree must be an integer.
  \item
    $\deg(\mathscr L \otimes \mathscr N) \geq 0$. Indeed, if
    $\deg(\mathscr L \otimes \mathscr N) < 0$, then $\mu \equiv 0$. Thus $\nu$ is
    nonvanishing and $\mathscr N \cong \mathcal O_{\mathscr C}$. Then we have
    $\deg(\mathscr L) < 0$, and thus $\varphi_i \equiv 0$ for all $i$. This
    contradicts to the assumption that $(\varphi,\mu)$ is nonvanishing.
  \item
    $\deg(\mathscr N) \geq 0$. Indeed, if $\deg(\mathscr N) < 0$,
    then $\nu \equiv 0$ and $\rho$ is nonvanishing. 
    Since $\mathrm{Aut}(\mathscr C, \Sigma^{\mathscr C})$ is infinite, this implies that
    $\deg(\mathscr L) = \frac{1}{5} \deg(\omega_{\mathscr C}^{\log}) \leq 0$. This contradicts to (iii).
  \item
    $\deg(\mathscr L \otimes \mathscr N) = 0$. Indeed, otherwise by (ii) and (iii) we
    must have $\deg(\mathscr L \otimes \mathscr N) \geq 1$. Then using (i) and 
    $\deg(\omega^{\log}_{\mathscr C}) \geq -2$, we must have $\deg(\mathscr
    N) <0$, contradicting to (iv).
  \end{enumerate}
  Since $\mathrm{Aut}(\mathscr C, \Sigma^{\mathscr C})$ is infinite, we are in one of the following three cases:
  \paragraph{Case 1: $\deg(\omega_{\mathscr C}^{\log}) = 0$}
  In this case, $\omega_{\mathscr C}^{\log}$ is trivial and any trivialization
  is fixed by $\mathrm{Aut}(\mathscr C, \Sigma^{\mathscr C})$.
  By (iv) and (v) above, \eqref{eq:negative-degree-appendix} implies
  $\deg(\mathscr L) = \deg(\mathscr N) = 0$.
  Thus each field is either nonvanishing or identically zero. Hence it is easy
  to see that $\xi$ has infinitely many automorphisms.

  \paragraph{Case 2: $\deg(\omega_{\mathscr C}^{\log}) < 0 $ and $\deg(\mathscr
    L^{\otimes (-5)} \otimes \omega_{\mathscr C}^{\log}) \geq 0$}
  Combining (i) and (v), we have $\deg(\mathscr
  L^{\otimes (-5)} \otimes \omega_{\mathscr C}^{\log}) = 0$. 
  Then $\deg(\mathscr L) \not\in \mathbb Z$, and hence $\Sigma^{\mathscr C} \neq
  \emptyset$. Then $\Sigma^{\mathscr C}$ is a single orbifold point and
  $\deg(\mathscr L) = -\deg(\mathscr N) = - \frac{1}{5}$, using (v). Thus $\varphi \equiv
  0$ and $\mu$ is an isomorphism $\mathscr L \otimes \mathscr N \cong
  \mathcal O_{\mathscr C}$. This way, $\xi$ is equivalent to the datum
  \[
    (\mathscr C, \Sigma^{\mathscr C}, \mathscr L, \rho, \nu),
  \]
  where $\rho \in H^0(\mathscr C, \mathscr L^{\otimes (-5)} \otimes
  \omega_{\mathscr C}^{\log})$ and
  $\nu \in H^0(\mathscr C, \mathscr L^{\otimes (-1)})$.
  Note that for degree reason, $\nu$ must have some zeros and $\rho$ is nonvanishing. If $\nu\equiv
  0$, then any automorphism of $(\mathscr C, \Sigma^{\mathscr C})$ lifts to an
  automorphism of $\xi$. Otherwise, it is easy to see that an automorphism
  of $(\mathscr C, \Sigma^{\mathscr C})$ lifts to an automorphism of $\xi$ if
  and only if it fixes the rational section $\rho/\nu^{5}$ of $\omega_{\mathscr
    C}^{\log}$. Choose coordinate $z$ on the coarse moduli of $\mathscr C$,
  which is isomorphic to $\mathbb P^1$, such that the image of $\Sigma^{\mathscr
    C}$ is at infinity. Then $\rho/\nu^{5}$ is a constant multiple of the pullback of
  $\operatorname{d}\!z$.
  Hence it is left invariant by any automorphism of $\mathscr C$ that induces $z
 \mapsto z + a, a\in \mathbb C$ on the coarse moduli. Hence $\mathrm{Aut}(\xi)$
  is always infinite.
  \paragraph{Case 3: $\deg(\omega_{\mathscr C}^{\log}) < 0, \deg(\mathscr
    L^{\otimes (-5)} \otimes \omega_{\mathscr C}^{\log}) < 0$}
  Thus $\rho \equiv 0$ and $\nu$ is nonvanishing. Thus $\mathscr N\cong \mathcal
  O_{\mathscr C}$ and $\xi$ is equivalent to
  \[
    (\mathscr C, \Sigma^{\mathscr C}, \mathscr L, \varphi_1 ,\ldots, \varphi_5, \mu).
  \]
  This is equivalent to a map to $\mathbb P^5$ of degree $\deg(\mathscr L) \geq
  0$. But by (iv) and (v) we have $\deg(\mathscr L) \leq 0$. Hence the map
  has degree zero. Hence $\mathrm{Aut}(\xi)$ is infinite.
\end{proof}

\bibliographystyle{alpha}
\bibliography{references.bib}
\end{document}